\documentclass[notitlepage,leqno,10pt]{article}

\usepackage{latexsym}
\usepackage{amsmath}
\usepackage{amsthm}
\usepackage{amsfonts}
\usepackage{amssymb}
\usepackage{graphicx}
\renewcommand{\a }{\alpha }

\renewcommand{\d}{\delta }

\newcommand{\e }{\epsilon }

\newcommand{\rh }{\rho }

\renewcommand{\th }{\theta }

\newcommand{\intbar}{\mathop{\int\makebox(-13.5,0){\rule[4pt]{.7em}{0.3pt}}%
\kern-6pt}\nolimits}

\newcommand{\be}{\begin{equation}}
\newcommand{\ee}{\end{equation}}

\newenvironment{pfnb}{\noindent{\bf Proof}}{\rule{2mm}{2mm}\medskip}

\newcommand{\R}{\mathbb{R}}
\newcommand{\Rn}{\mathbb{R}^n}

\newcommand{\Rtre}{\mathbb{R}^3}

\newcommand{\N}{\mathbb{N}}

\begin{document}

\author{Andrea MONDINO$^{1}$}

\date{}

\title{The conformal Willmore Functional: a perturbative approach}

\newtheorem{lem}{Lemma}[section]
\newtheorem{pro}[lem]{Proposition}
\newtheorem{thm}[lem]{Theorem}
\newtheorem{rem}[lem]{Remark}
\newtheorem{cor}[lem]{Corollary}
\newtheorem{df}[lem]{Definition}
\newtheorem*{Theorem}{Theorem}
\newtheorem*{Lemma}{Lemma}
\newtheorem*{Proposition}{Proposition}
\newtheorem*{claim}{Claim}

\maketitle

\begin{center}

{\small SISSA, via Beirut 2-4, 34014 Trieste, Italy.}

\

\end{center}

\footnotetext[1]{E-mail address: mondino@sissa.it}

\

\

\begin{center}
\noindent {\sc abstract}. The conformal Willmore functional (which is conformal invariant in general Riemannian manifold $(M,g)$) is studied with a perturbative method: the Lyapunov-Schmidt reduction. Existence of critical points is shown in ambient manifolds $(\Rtre, g_\e)$ -where $g_\e$ is a metric close and asymptotic to the euclidean one. With the same technique a non existence result is proved in general Riemannian manifolds $(M,g)$ of dimension three. 
\bigskip\bigskip


\bigskip

\centerline{\bf AMS subject classification: }
53C21, 53C42, 58E99, 35J60.
\end{center}

\section{Introduction}\label{s:in}

The aim of this paper is to study a (Riemannian) conformally invariant Willmore functional. The study of Conformal Geometry was started by H. Weil and E. Cartan in the beginning of the XX century and since its foundation it has been playing ever more a central role in Riemannian Geometry; its task is to analyze how geometric quantities change under conformal transformations (i.e. diffeomorphisms which preserves angles) and possibly find out conformal invariants (i.e. quantities which remain unchanged under conformal transformations). 

Let us first recall the definition of ``standard'' Willmore functional for immersions in $\Rtre$ which is a topic of great interest in the contemporary research (see for instance \cite{BK}, \cite{KS} and \cite{Riv}). Given a compact orientable Riemannian surface $(\mathring{M},\mathring{g})$ isometrically immersed in $\Rtre$ endowed with euclidean metric, the ``standard'' Willmore functional of $\mathring{M}$ is defined as
\begin{equation}\label{W(M)}
W(\mathring{M})=\int_{\mathring{M}} \frac{H^2}{4} d\Sigma
\end{equation}
where $H$ is the mean curvature and $d\Sigma$ is the area form of $(\mathring{M},\mathring{g})$ (we will always adopt the convention that $H$ is the sum of the principal curvatures: $H:=k_1+k_2$). 
\\This functional satisfies two crucial properties:

a) $W$ is invariant under conformal transformations of $\Rtre$; that is, given $\Psi:\Rtre \to \Rtre$ a conformal transformation, $W(\Psi(\mathring{M}))=W(\mathring{M})$ (Blaschke 1929-White 1973).

b) $W$ attains its strict global minimum on the standard spheres $S_p^\rh$ of $\Rtre$ (hence they form a critical manifold - i.e. a manifold made of critical points):
\begin{equation}\label{TeoWillmin}
W(\mathring{M}):=\int_{\mathring{M}} \frac{H^2}{4} d\Sigma \geq  4 \pi; \quad W(\mathring{M})=4 \pi \Leftrightarrow \mathring{M}=S_p^\rho.
\end{equation} 
The proofs of the last facts can be found in \cite{Will} (pag. 271 and pag. 276-279).

Clearly the ``standard'' Willmore functional $W$ can be defined in the same way for compact oriented surfaces immersed in a \emph{general Riemannian manifold} $(M,g)$ of dimension three. Although this functional has several interesting applications (see for instance the Introduction of \cite{Mon}), it turns out that $W$ is not conformally invariant. 

As proved by Bang-Yen Chen in \cite{Chen} (see also \cite{Wei} and for higher dimensional and codimensional analogues \cite{PW} ), the ``correct'' Willmore functional from the conformal point of view is defined as follows.
Given a compact orientable Riemannian surface $(\mathring{M},\mathring{g})$ isometrically immersed in the three dimensional Riemannian manifold $(M,g)$, the \emph{conformal Willmore functional} of $\mathring{M}$ is 
\begin{equation}
I(\mathring{M}):= \int_{\mathring{M}} \bigg(\frac{H^2}{4}-D \bigg) d\Sigma
\end{equation}
where $D:=k_1k_2$ is the product of the principal curvatures and as before $H$ and $d\Sigma$ are respectively the mean curvature and the area form of $(\mathring{M},\mathring{g})$. In the aforementioned papers it is proved that $I$ is conformally invariant (i.e. given $\Psi:(M,g) \to (M,g)$ a conformal transformation, $I(\Psi(\mathring{M}))=I(\mathring{M})$) so in this sense it is the ``correct'' generalization of the standard Willmore functional which, as pointed out, is conformally invariant in $\Rtre$. We say that $I$ generalizes $W$ because if $\Rtre$ is taken  as  ambient manifold, the quantity $D=k_1k_2$ is nothing but the Gaussian curvature which, fixed the topology of the immersed surface, gives a constant when integrated (by the Gauss-Bonnet Theorem) hence it does not influence the variational properties of the functional.
\\

A surface which makes the conformal Willmore functional $I$ stationary with respect to normal variations is called \emph{conformal Willmore surface} and it is well known ( the expression of the differential in full generality is stated without proof in \cite{PW} and the computations can be found in \cite{HL}, here we deal with a particular case which will be computed in the proof of Proposition \ref{lem:svilI'} ) that such a surface satisfies the following PDE:
$$ \frac{1}{2}\triangle_{\mathring{M}}H + H\bigg(\frac{H^2}{4}-D\bigg)+\frac{(\lambda_1-\lambda_2)}{2}[R(\mathring{N},e_1,\mathring{N},e_1)-R(\mathring{N},e_2,\mathring{N},e_2)]+\sum_{ij}(\nabla_{e_i}R)(\mathring{N},e_j,e_j,e_i)=0$$
where $\triangle_{\mathring{M}}$ is the Laplace Beltrami operator on $\mathring{M}$, $R$ is the Riemann tensor of the ambient manifold $(M,g)$ (for details see ``notations and conventions''), $\mathring{N}$ is the inward unit normal vector, $\lambda_1$ and $\lambda_2$ are the principal curvatures and   $e_1$, $e_2$ are the normalized principal directions.

The goal of this paper is to study the existence of conformal Willmore surfaces. 
\\The topic has been extensively studied in the last years: in \cite{ZG} the author generalizes the conformal Willmore functional to arbitrary dimension and codimension and studies the existence of critical points  in space forms; in \cite{HL} the authors compute the differential of $I$ in full generality and give examples of conformal Willmore surfaces in the sphere and in complex space forms; other existence results in spheres or in space forms are studied for instance in \cite{GLW}, \cite{LU}, \cite{WG} and \cite{MW}. 
\\

The novelty of this paper is that the conformal Willmore functional is analyzed in an ambient manifold with non constant sectional curvature: we will give existence (resp. non existence) results for curved metrics in $\Rtre$, close and asymptotic to the flat one (resp. in general Riemannian manifolds). More precisely, taken $h_{\mu\nu} \in C^\infty _0 (\Rtre)$ a smooth bilinear form with compact support (as we will remark later it is sufficient that $h_{\mu \nu}$ decreases fast at infinity with its derivatives) we take as ambient manifold
\begin{equation} \label{def:ge}
(\Rtre,g_\e) \quad \text{with} \quad g_\e=\delta+\e h 
\end{equation} 
where $\delta$ is the euclidean scalar product. 

The test surfaces are perturbed standard spheres (resp. perturbed geodesic spheres), let us define them.
Let $S_p^\rh$ be a standard sphere of $\Rtre$ parametrized by
$$\Theta \in S^2 \mapsto p+\rh \Theta$$
and let $w\in C^{4,\a}(S^2)$ be a small function, then the \emph{perturbed standard sphere} $S_p^\rh(w)$ is the surface parametrized as
$$\Theta \in S^2 \mapsto p+\rh (1-w(\Theta)) \Theta.$$ 
Analogously the \emph{perturbed geodesic sphere} $S_{p,\rh}(w)$ is the surface parametrized by 
$$\Theta \in S^2 \mapsto Exp_p [\rh(1-w(\Theta)) \Theta] $$
where $S^2$ is the unit sphere of $T_pM$, $Exp_p$ is the exponential map centered at $p$ and, as before, $w$ is a small function in $C^{4,\a}(S^2)$. 

The main results of this paper are Theorem \ref{teoExist} and Theorem \ref{teoExist2} below, which will be proved in Subsection \ref{subsec:ProofEx}.
Before stating them recall that given a three dimensional Riemannian manifold $(M,g)$, the \emph{traceless Ricci tensor} $S$ is defined as 
\begin{equation} \label{def:S}
S_{\mu\nu}:=R_{\mu\nu}-\frac{1}{3}g_{\mu\nu}R
\end{equation}
where $R_{\mu\nu}$ is the Ricci tensor and $R$ is the scalar curvature. Its squared norm at a point $p$ is defined as $\|S_p\|^2=\sum_{\mu,\nu=1}^3 S_{\mu\nu}(p)^2$ where $S_{\mu\nu}(p)$ is the matrix of $S$ at $p$ in an orthonormal frame. Expanding in $\e$ the curvature tensors (see for example \cite{And-Mal} pages 23-24) it is easy to see that the traceless Ricci tensor corresponding to $(\Rtre,g_\e)$ (defined in \eqref{def:ge} )is
\begin{equation}\label{def:s}
\|S_p\|^2= \e^2 \tilde{s}_p+o(\e^2)
\end{equation}
where $\tilde{s}_p$ is a nonnegative quadratic function in the second derivatives of $h_{\mu\nu}$ which does not depend on $\e$. 
In the following Theorem, $\pi$ will denote an affine plane in $\Rtre$ and  $H^1(\pi)$ will be the Sobolev space of the $L^2$ functions defined on $\pi$ whose  distributional gradient is a vector valued $L^2$ integrable function.  $H^1(\pi)$ is equipped with the norm
$$\|f\|_{H^1(\pi)}^2:=\|f\|_{L^2(\pi)}^2+\|\nabla f\|_{L^2(\pi)}^2 \quad \forall f\in H^1(\pi).$$ 

Now we can state the Theorems.

\begin{thm}\label{teoExist}
Let $h \in C^{\infty}_0(\Rtre)$ be a symmetric bilinear form with compact support and let $c$ be such that 
$$c:=\sup\{\|h_{\mu\nu}\|_{H^1(\pi)}:\pi \text{ is an affine plane in $\Rtre$,  } \mu,\nu=1,2,3\}.$$
Then there exists a constant $A_c>0$  depending on $c$ with the following property: if there exists a point $\bar{p}$ such that 
$$\tilde{s}_{\bar{p}} >A_c  $$
then, for $\e$ small enough, there exists a perturbed standard sphere $S_{p_\e}^{\rh_\e}(w_\e)$ which is a critical point of the conformal Willmore functional $I_\e$ converging to a standard sphere as $\e\to 0$.
\end{thm} 

It is well-known  (see Remark \ref{rem:NoEx} point $3$) that if a three dimensional Riemannian Manifold has non constant sectional curvature then the traceless Ricci tensor $S$ cannot vanish everywhere. Clearly $(\Rtre, g_\e)$ has non constant sectional curvature (the metric is asymptotically flat but not flat) hence it cannot happen that $\|S\|^2\equiv 0$; for the following existence result we ask that this non null quantity has non degenerate expansion in $\e$: we assume 
\begin{equation}\label{def:M}
M:=\max _{p\in \Rtre} \tilde{s}_p>0.
\end{equation}
Actually it is a maximum and not only a supremum because the metric is asymptotically flat. 

The following is like a mirror Theorem to the previous existence result: in the former we bounded $c$ and asked $\tilde{s}$ to be large enough at one point, in the latter we assume that $\tilde{s}$ is non null at one point (at least) and we ask $c$ to be small enough.

\begin{thm}\label{teoExist2}
Let $h$, $c$ be as in Theorem \ref{teoExist} and $M$ satisfying \eqref{def:M}. There exists  $\delta_M>0$ depending on $M$ such that if $c<\delta_M$ then, for $\e$ small enough,  there exists a perturbed standard sphere $S_{p_\e} ^{\rh_\e} (w_\e)$ which is a critical point of the conformal Willmore functional $I_\e$ converging to a standard sphere as $\e\to 0$.
\end{thm}

\begin{rem}\label{rem:Ex}
\begin{enumerate}
	\item As done in \cite{Mon}, the assumption  $h\in C_0 ^\infty(\Rtre)$ in Theorem \ref{teoExist} and Theorem \ref{teoExist2} can be relaxed asking that $h$ decreases fast enough at infinity with its derivatives.
	\item The conditions of Theorem \ref{teoExist}
	$$\sup\{\|h_{\mu\nu}\|_{H^1(\pi)}:\pi \text{ is an affine plane in $\Rtre$,  } \mu,\nu=1,2,3\}\leq c$$
	and 
	$$\tilde{s}_{\bar{p}} >A_c$$
	are compatible. In fact the former involves only the first derivatives of $h$ while the latter the second derivatives (see for instance \cite{And-Mal} page 24). Of course the same fact is true for the conditions  $\tilde{s}_{\bar{p}} \geq M$ and $c<\delta_M$ of Theorem 
	\ref{teoExist2}.
	\item If the perturbation $h$ satisfies some symmetries (invariance under reflections or rotations with respect to planes, lines or points of $\Rtre$), it is possible to prove multiplicity results (see Subsection 5.2 of \cite{Mon}). 
	\item If $h$ is $C^\infty$  then a standard regularity argument (see the paper of Leon Simon \cite{SiL} pag. 303 or the book by Morrey \cite{MCB}) shows that a $C^{2,\a}$ conformal Willmore surface is actually  $C^\infty$. It follows that the conformal Willmore surfaces exhibited in the previous Theorems, which are $C^{4,\a}$ by construction, are  $C^\infty$.
	\item The critical points $S_{p_\e}^{\rh_\e}(w_\e)$ of $I_\e$ are of (maybe degenerate) saddle type. In fact from \eqref{TeoWillmin} the standard spheres $S_p^\rh$ are strict global minima in the direction of variations in $C^{4,\a}(S^2)^\perp = Ker[I_0''(S_p^\rh)]^\perp \cap C^{4,\a}(S^2)$, it is easy to see that for small $\e$ the surfaces $S_{p_\e}^{\rh_\e}(w_\e)$ are still minima in the $C^{4,\a}(S^2)^\perp$ direction; but, since they are obtained as maximum points of the reduced functional, in the direction of $Ker[I_0''(S_p^\rh)]$ they are (maybe degenerate) maximum points. 
\end{enumerate}
\end{rem}

As we said before, the non existence result concerns perturbed geodesic spheres of small radius. Let us state it:

\begin{thm}\label{TeoNoExist}
Let $(M,g)$ be a  Riemannian manifold. Assume that the traceless Ricci tensor of $M$ at the point $\bar{p}$ is not null:
 $$\|S_{\bar{p}}\|\neq 0.$$
Then there exist $\rho_0>0$ and $r>0$ such that for radius $\rho<\rho_0$ and perturbation $w\in C^{4,\alpha}(S^2)$ with $\|w\|_{C^{4,\a}(S^2)}<r$, the surfaces $S_{\bar{p},\rho}(w)$ are \emph{not} critical points of the conformal Willmore functional $I$. 
\end{thm}

\begin{rem}\label{rem:NoEx}
\begin{enumerate}
	\item Observe the difference with the flat case: thanks to \eqref{TeoWillmin}, in $\Rtre$ the spheres of any radius are critical points of the conformal Willmore functional $I$ (has we noticed, the term $D$ does not influence the differential properties of the functional by Gauss-Bonnet Theorem); on the contrary, in the case of ambient metric with non null traceless Ricci curvature we have just shown that the geodesic spheres of small radius are not critical points.
	\item The condition $\|S_p\|\neq 0$ is generic.
	\item If $(M,g)$ has not constant sectional curvature then there exists at least one point $\bar{p}$ such that  $\|S_ {\bar{p}}\| \neq 0$. In fact if $\|S\|\equiv 0$ then $(M,g)$ is Einstein, but the Einstein manifolds of dimension three have constant sectional curvature (for example see \cite{Pet} pages 38-41). 
\end{enumerate}
\end{rem}

The abstract method employed throughout the paper is similar to the one  used in the previous article \cite{Mon}: the Lyapunov-Schmidt reduction (for more details about the abstract method see Section \ref{subsec:PertMeth}).
\\The main difficulty here is that, as we will see, the expansions are degenerate, and require more precision.

We discuss next the structure of the article, but first let us explain (informally) the main idea (for the details see Subsection \ref{subsec:FDRnon} and Subsection \ref{subsec:FDR}). 
\\As we remarked, \eqref{TeoWillmin} implies that the Willmore functional in the euclidean space $\Rtre$ possesses a critical manifold $Z$  made of standard spheres $S_p^\rh$. The tangent space to $Z$ at $S_p^\rh$ is composed of constant and affine functions on $S_p^\rh$ so, with a pull back via the parametrization, on $S^2$. The second derivative of $I_0$ at $S_p^\rh$ is 
$$I_0''(S_p^\rh)[w]=\frac{1}{2} \triangle_{S^2} (\triangle_{S^2}+2)w$$
(for explanations and details see Remark \ref{rem:I0''}) which is a Fredholm operator of index zero and whose Kernel is made of the constant and affine functions; exactly the tangent space to $Z$.
\\So, considered $C^{4,\a}(S^2)$ as a subspace of $L^2(S^2)$ and called 

$$C^{4,\alpha}(S^2)^\perp :=C^{4,\alpha}(S^2)\cap Ker[\triangle_{S^2}(\triangle_{S^2}+2)]^\perp,$$
it follows that $I_0''|_{C^{4,\alpha}(S^2)^\perp}$ is invertible on its image and one can apply the Lyapunov-Schmidt reduction. 
Thanks to this reduction, the critical points of $I_\e$ in a neighbourhood of $Z$ are exactly the stationary points of a function (called reduced functional) $\Phi_\e:Z \to \R$ of \emph{finitely many} variables (we remark that in a neighbourhood of $Z$ the condition is necessary and sufficient for the existence of critical points of $I_\e$). 
\\

In order to study the function $\Phi_\e$,  we will compute explicit formulas  and estimates of the conformal Willmore functional.
More precisely for small radius $\rho$ we will give an expansion of the functional on small perturbed geodesic spheres, for  large radius we will estimate the functional on perturbed standard spheres and we will link the geodesic and standard spheres in a smooth way using a cut off function (for details see Subsection \ref{subsec:FDR}).

The paper is organized as follows: in Section \ref{sec:IMg} we will start  in the most general setting, the conformal Willmore functional for small perturbed geodesic spheres in ambient manifold $(M,g)$. Even in this case the reduction method can be performed, using the small radius $\rh$ as perturbation parameter (see Lemma \ref{lemEqAuxContr}).
\\Employing the geometric expansions of Subsection \ref{subsec:geomExp} and the expression of the constrained $w$ given in Subsection \ref{subsec:FDRnon}, in Subsection \ref{subsec:ExpIMg} we will compute the expansion of the reduced functional on small perturbed geodesic spheres of $(M,g)$. Explicitly, in Proposition \ref{svilPhipr}, we will get 
\begin{equation}\label{ExpSmallr}
\Phi(p,\rh)=\frac{\pi}{5}\|S_p\|^2\rh^4+O_p(\rh^5)
\end{equation}
where $\Phi(.,.)$ is the reduced functional and, as before, $S_p$ is the traceless Ricci tensor evaluated at $p$.
Using this formula we will show that if $\|S_{\bar{p}}\|\neq 0$ then $\Phi(\bar{p},.)$ is strictly increasing for small radius. 
The non existence result will follow from the necessary condition.

Section \ref{sec:Ie} will be devoted to the conformal Willmore functional in ambient manifold $(\Rtre, g_\e)$.   
In Subsection \ref{subsec:FDR} we will treat the applicability of the abstract method and in the last Subsection \ref{subsec:ProofEx} we will bound the reduced functional $\Phi_\e$ for large radius $\rh$ using the computations of Subsection \ref{subsec:ExpRF}. We remark that the expansion of $\Phi_\e$ is degenerate in $\e$ (i.e. the first term in the expansion is null and $\Phi_\e=O(\e^2)$), clearly this feature complicates the problem.
Using the estimates on the reduced functional $\Phi_\e$ for large radius and the expansions for small radius (since for small radius we take geodesic spheres it will be enough to specialize \eqref{ExpSmallr} in the setting $(\Rtre,g_\e)$) we will force $\Phi_\e$ to have a global maximum, sufficient condition to conclude with the existence results.     
 
\begin{center}

{\bf Notations and conventions}

\end{center}

\noindent

1) $\R^+$ denotes the set of strictly positive real numbers.
\\

2) As mentioned in the Introduction, the perturbed spheres will play a central role throughout this paper.

$\cdot$ First, let us define the perturbed standard sphere $S_p^\rho(w) \subset \Rtre$ we will use to prove the existence results.
We denote with $S^2$ the standard unit sphere in the euclidean $3$-dimensional space , $\Theta \in S^2$ is the radial versor with components $\Theta^\mu$  parametrized by the polar coordinates $0<\theta^1<\pi$ and  $0<\theta^2< 2\pi$ chosen in order to satisfy
\begin{displaymath} 
\left\{ \begin{array}{ll}
\Theta^1= \sin \theta ^1 \cos \theta ^2 \\
\Theta^2= \sin \theta ^1 \sin \theta ^2 \\
\Theta^3= \cos \theta ^1. \\
\end{array} \right.
\end{displaymath}
We call  $\Theta_i$ the coordinate vector fields on $S^2$ 
$$\Theta_1:= \frac{\partial \Theta}{\partial \theta^1}, \quad \Theta_2:= \frac{\partial \Theta}{\partial \theta^2}$$
and $\bar{\theta}_i$ or $\bar{\Theta}_i$ the corresponding normalized ones 
$$\bar{\theta}_1=\bar{\Theta}_1:=\frac{\Theta_1}{\|\Theta_1\|}, \quad \bar{\theta_2}=\bar{\Theta}_2:=\frac{\Theta_2}{\|\Theta_2\|}.$$
The standard sphere in $\Rtre$ with center $p$ and radius $\rh>0$ is denoted by $S_p^\rh$; we parametrize it as $(\theta^1,\theta^2) \mapsto p+\rho\Theta(\th^1,\th^2)$ and  call $\theta_i$ the coordinate vector fields 
$$\theta_1:=\rho \frac{\partial \Theta}{\partial \theta^1}, \quad \theta_2:=\rho \frac{\partial \Theta}{\partial \theta^2}.$$

The perturbed spheres will be normal graphs on standard spheres by a function $w$ which belongs to a suitable function space. 
Let us introduce the function space which has been chosen by technical reasons (to apply Schauder estimates in Lemma \ref{lemw}).
\\Denote $C^{4,\a}(S^2)$ (or simply $C^{4,\a}$)  the set of the $C^4$ functions on $S^2$ whose fourth derivatives, with respect to the tangent vector fields, are $\a$-H\"older ($0<\a<1$). The Laplace-Beltrami operator on $S^2$ is denoted by $\triangle_{S^2}$ or, if there is no confusion, as $\triangle$.
The fourth order elliptic operator $\triangle (\triangle +2)$ induces a splitting of $L^2(S^2)$:
$$L^2(S^2)= Ker[\triangle (\triangle+2)]\oplus Ker[\triangle (\triangle+2)]^\perp $$
(the splitting makes sense because the kernel is finite dimensional, so it is closed).
\\If we consider $C^{4,\a}(S^2)$ as a subspace of $L^2(S^2)$, we can define 
$$C^{4,\alpha}(S^2)^\perp :=C^{4,\alpha}(S^2)\cap Ker[\triangle (\triangle+2)]^\perp.$$
Of course ${C^{4,\alpha}(S^2)}^\perp$ is a Banach space with respect to the $C^{4,\a}$ norm; it is the space from which we will get the perturbations $w$. If there is no confusion $C^{4,\alpha}(S^2)^\perp$ will be called simply ${C^{4,\a}}^\perp$.
\\Now we can define the perturbed spheres we will use to prove existence of critical points: 
fix $\rh>0$ and a small ${C^{4,\a}}^\perp$ function $w$; the perturbed sphere $S_p^\rh(w)$ is the surface parametrized by
$$\Theta \in S^2 \mapsto p+ \rho\big(1-w(\Theta)\big) \Theta.$$

$\cdot$ Now let us define the perturbed geodesic spheres $S_{p,\rh}(w)$ in the three dimensional Riemannian manifold $(M,g)$; we will use them to prove the non-existence result. 
\\Once a point $p\in M$ is fixed we can consider the exponential map 
$Exp_p$ with center $p$. For $\rh>0$ small enough, the sphere $\rh S^2\subset T_pM$ is contained in the radius of injectivity of the exponential. We  call $S_{p,\rh}$ the geodesic sphere of center $p$ and radius $\rh$. This hypersurface can be parametrized by
$$\Theta \in S^2 \subset T_pM \mapsto Exp_p[\rh \Theta].$$
Analogously to the previous case, fix $p\in M$, $\rh>0$ and a small $C^{4,\a}(S^2)$ function $w$; the perturbed geodesic sphere $S_{p,\rh}(w)$ is the surface parametrized by
$$\Theta \in S^2 \mapsto Exp_p[\rho\big(1-w(\Theta)\big) \Theta].$$
The tangent vector fields on $S_{p,\rh}(w)$ induced by the canonical polar coordinates on $S^2$ are denoted by $Z_i$.
\\

3) Let $(M,g)$ be a $3$-dimensional Riemannian manifold. 

$\cdot$ First we make the following convention: the Greek index letters, such as $\mu,\nu,\iota,\ldots,$ range from $1$ to $3$ while the Latin index letters, such as $i,j,k,\ldots,$ will run from $1$ to $2$.   

$\cdot$ About the Riemann curvature tensor we adopt the convention of \cite{Will}:
denoting $\mathfrak X(M)$ the set of the vector fields on $M$, $\forall X,Y,Z \in \mathfrak X (M)$
$$R(X,Y)Z := \nabla _X \nabla_Y Z - \nabla_Y \nabla_X Z - \nabla_{[X,Y]}Z $$  
$$R(X,Y,Z,W) :=g(R(Z,W)Y,X);$$
chosen in $p$ an orthonormal frame  $E_\mu$, the Ricci curvature tensor is
\begin{eqnarray}
Ric_p(v_1,v_2)&:=&\sum_{\mu = 1} ^{3} R(E_\mu,v_1,E_\mu,v_2)=\sum_1 ^{3} g(R_p(E_\mu,v_2)v_1, E_\mu) \nonumber \\
&=&-\sum_{\mu = 1} ^{3} g(R_p(v_2,E_\mu)v_1, E_\mu) \quad \forall v_1,v_2 \in T_pM. \label{eqRic}
\end{eqnarray} 

$\cdot$ In order to keep formulas not too long, we introduce the following notation:
\begin{eqnarray}
R(0i0j)&:=& g(R_p(\Theta ,\Theta _i)\Theta ,\Theta _j) \nonumber\\ 
\nabla_0R(0i0j)&:=& g(\nabla_{\Theta}R_p(\Theta ,\Theta _i)\Theta ,\Theta _j)\nonumber\\ 
\nabla_{00} R(0i0j)&:=& g(\nabla_{\Theta}\nabla_{\Theta}R_p(\Theta ,\Theta _i)\Theta ,\Theta _j) \nonumber\\
R(0i0\mu)&:=&g(R_p(\Theta ,\Theta _i)\Theta ,E_\mu).\nonumber
\end{eqnarray}
In the following ambiguous cases we will mean:
\begin{eqnarray}
R(0101)&:=& g(R_p(\Theta ,\Theta _1)\Theta ,\Theta _1) \nonumber\\
R(0\bar{2}0\bar{2})&:=& g(R_p(\Theta ,\bar{\Theta} _2)\Theta ,\bar{\Theta} _2)  \nonumber\\ 
R(010\bar{2})&:=&g(R_p(\Theta ,\Theta _1)\Theta ,\bar{\Theta} _2) \nonumber.
\end{eqnarray}
$\cdot$Recall the definitions of the Hessian and the Laplace-Beltrami operator on a function $w$:
$$Hess(w)_{\mu\nu}:=\nabla _\mu \nabla_\nu w$$
$$\triangle:= g^{\mu\nu} \nabla _\mu \nabla_\nu w. $$

4) Let $(\mathring{M},\mathring{g})\hookrightarrow (M,g)$ be an isometrically immersed surface. Recall the notion of second fundamental form $\mathring{h}$: fix a point $p$ and an orthonormal base $Z_1,Z_2$ of $T_p\mathring{M}$; the (inward) normal unit vector is denoted as $\mathring{N}$. By the Weingarten equation  $\mathring{h}_{ij}= - g( \nabla _{Z_i} \mathring{N}, Z_j)$.
\\ Call $k_1$ and $k_2$ the principal curvatures (the eigenvalues of the second fundamental form with respect to the first fundamental form of $\mathring{M}$, i.e. the roots of $\det(\mathring{h}_{ij}-k\mathring{g}_{ij})=0$). We adopt the convention that the mean curvature is defined as $H:=k_1+k_2$. 
\\The product of the principal curvatures will be denoted with $D$:
\begin{equation} \label{defD} 
D:=k_1k_2=\frac {det(\mathring {h})}{det( \mathring {g})}.
\end{equation}

5) $\cdot$ Following the notation of \cite{PX}, given $a \in \N$, any expression of the form $L_p^{(a)}(w)$ denotes a linear combination of the function $w$ together 
with its derivatives with respect to the tangent vector fields $\Theta_i$ up to order $a$. The coefficients of $L_p^{(a)}$ 
might depend on $\rh$ and $p$ but, for all $k \in \N$, there exists a constant $C>0$ independent on $\rh \in (0,1)$ and $p \in M$ such that 
$$\|L_p^{(a)}(w)\|_{C^{k,\a}(S^2)}\leq C \|w \|_{C^{k+a,\a}(S^2)}.$$  

$\cdot$ Similarly, given $b\in \N$, any expression of the form $Q_p^{(b)(a)}(w)$ denotes a nonlinear operator in the function $w$ together with its derivatives with respect to the tangent vector fields $\Theta_i$ up to order $a$ such that, for all $p \in M$, $Q_p^{(b)(a)}(0)=0$. The coefficients of the Taylor expansion of $Q_p^{(b)(a)}(w)$ in powers of $w$ and its partial derivatives might depend on $\rh$ and $p$ but, for all $k \in \N$, there exists a constant $C>0$ independent on $\rh \in (0,1)$ and $p \in M$ such that 
\begin{equation}\label{stimaQ}
\|Q_p^{(b)(a)}(w_2)-Q_p^{(b)(a)}(w_1) \|_{C^{k,\a}(S^2)}\leq c \big(\|w_2\|_{C^{k+a,\a}(S^2)}+\|w_1\|_{C^{k+a,\a}(S^2)}\big)^{b-1}
\times \|w_2-w_1\|_{C^{k+a,\a}(S^2)},
\end{equation}
provided $\|w_l\|_{C^{a}(S^2)}\leq 1$, $l=1,2$. If the numbers $a$ or $b$ are not specified, we intend that their value is $2$.

$\cdot$ We also agree that any term denoted by $O_p(\rho ^d)$ is a smooth function on $S^2$ that might depend on $p$ but which is bounded by a constant (independent on $p$) times $\rho ^d$ in $C^k$ topology, for all $k \in N$.
\\

6) Large positive constants are always denoted by $C$, and the value of
$C$ is allowed to vary from formula to formula and also within the
same line. When we want to stress the dependence of the constants on
some parameter (or parameters), we add subscripts to $C$, as $C_\d$,
etc.. Also constants with subscripts are allowed to vary.

\begin{center}

{\bf Acknowledgments}

\end{center}

\noindent 
This work has been supported by M.U.R.S.T within the PRIN 2006 ``Variational Methods and Nonlinear Differential Equations'' and by the Project FIRB-IDEAS ``Analysis and Beyond''.
\\I would like to thank my supervisor Prof. Malchiodi for support and for useful discussions about this topic.

\section{A Preliminary result: the Lyapunov-Schmidt reduction}\label{subsec:PertMeth}

The technique used throughout this paper relies on an abstract perturbation method which first appeared in \cite{AB1}, \cite{AB2} and is extensively treated with proofs and examples in \cite{AM}.
Let us briefly summarize it. Actually we present the abstract method in a form which permits to deal with degenerate expansions (as the ones we will have to handle).  

Given an Hilbert space $H$, let $I_\e: H \to \R$ be a $C^2$ functional of the form
$$I_\e(u)=I_0(u)+\e G_1(u)+\e^2 G_2(u)+o(\e^2),$$
where $I_0 \in C^2 (H, \R)$ plays the role of the unperturbed functional and $G_1, G_2 \in C^2(H,\R)$ are the perturbations.

We first assume  that there exists a finite dimensional smooth manifold $Z$ made of critical points of $I_0$:
 $I_0'(z)=0$ for all $z\in Z$. The set $Z$ will be called \emph{critical manifold} (of $I_0$).  
The critical manifold is supposed to satisfy the following non degeneracy conditions:

(ND) for all $z\in Z$,  $T_zZ = Ker[I_0''(z)]$, 

(Fr) for all $z\in Z$, $I_0''(z)$ is a Fredholm operator of index zero.
\\Under these assumptions it is known that near $Z$ there exists a perturbed manifold $Z^\e$ such that the critical points of $I_\e$ constrained on $Z^\e$ give rise to stationary points of $I_\e$. 
\\More precisely, the key result is the following Theorem.

\begin{thm}\label{Thm:FDR}
Suppose $I_0$ possesses a non degenerate (satisfying (ND) and (Fr)) critical manifold $Z$ of dimension $d$.
\\Given a compact subset $Z_c$ of $Z$, there exists $\e_0>0$ such that for all $|\e|<\e_0$ there is a smooth function 
$$w_\e(z) : Z_c \to H$$
such that

(i) for $\e=0$ it results $w_\e (z)=0$,  $\forall z \in Z_c$;

(ii) $w_\e(z)$ is orthogonal to $T_zZ$, $\forall z \in Z_c$;

(iii) the manifold
$$Z^\e=\{z+w_\e(z): z\in Z_c \}$$
is a natural constraint for $I_\e'$. Namely, denoting 
$$\Phi_\e(z)= I_\e(z+w_\e(z)):Z_c \to \R $$
the constriction of $I_\e$ to $Z^\e$, if $z_\e$ is a critical point of $\Phi_\e$ then $u_\e=z_\e+w_\e(z_\e)$ is a critical point of $I_\e$.
\end{thm} 
Thanks to this fundamental tool, in order to find critical points of $I_\e$, we can reduce ourselves to study $\Phi_\e$ which is a function of \emph{finitely} many variables. 

If we are slightly more accurate, it can be shown that the function $w_\e(z)$ is of order $O(\e)$ as $\e\to 0$ uniformly in $z$ varying in the compact $Z_c$. In our application, the expansion is degenerate in the sense that 
$$G_1(z)=0 \quad \forall z \in Z.$$
Using the previous facts, by a Taylor expansion it is  easy to see that (we will prove it in full detail in Lemma \ref{lemsvile2}) 
$$\Phi_\e(z)= \e^2 \Big[G_2(z)-\frac{1}{2} \big(G_1'(z)|I_0''(z)^{-1} G_1'(z) \big)  \Big]+o(\e^2).$$
In Section \ref{sec:Ie} we will give sense to this formula, which will be crucial for the estimates involved in the existence result.

\section {The conformal Willmore functional on perturbed geodesic spheres $S_{p,\rho}(w)$ of a general Riemannian Manifold $(M,g)$}\label{sec:IMg}

\subsection{Geometric expansions}\label{subsec:geomExp}
In this subsection we give accurate expansions of the geometric quantities appearing in the conformal Willmore functional. First we recall and refine the well-known expansions of the first and  second fundamental form and the mean curvature for the 
geodesic perturbed spheres $S_{p,\rh}(w)$ introduced in the previous ``notations and conventions''. Recall that $\Theta_i$ are the coordinate vector fields on $S^2$ (induced by polar coordinates) and $Z_i$ are the corresponding coordinate vector fields on $S_{p,\rh}(w)$. The derivatives of $w$ with respect to $\Theta_i$ are denoted by $w_i$.

Let $\mathring{g}$ denote the first fundamental form on $S_{p,\rh}(w)$ induced by the immersion in $(M,g)$. The next Lemma, whose proof can be found in \cite{PX} (Lemma 2.1), gives an expansion of the components $\mathring{g}_{ij}:=g_p(Z_i,Z_j)$:

\begin{lem}\label{lem1forma}
The first fundamental form on $S_{p,\rh}(w)$ has the following expansion:
\begin{eqnarray}
(1-w)^{-2} \rho^{-2} \mathring{g}_{ij}&=& g(\Theta _i,\Theta _j) + (1-w)^{-2}w_i w_j +\frac {1}{3} R(0i0j) \rho ^2 (1-w)^2  +\frac {1}{6} \nabla_0 R(0i0j) \rho ^3 (1-w)^3 \nonumber\\
&&+\Big[ \frac {1}{20} \nabla _{00} R(0i0j)+\frac {2}{45} R(0i0\mu) R(0j0\mu)\Big]\rho ^4 (1-w)^4+ O_p(\rho^5)+\rho^5 L_p(w)+\rho^5 Q_p^{(2)}(w) \nonumber
\end{eqnarray}
where all curvature terms and scalar products are evaluated at $p$ (since we are in normal coordinates, at $p$ the metric is euclidean).
\end{lem} 

Let $\mathring{h}$ denote the second fundamental form on $S_{p,\rh}(w)$ induced by the immersion in $(M,g)$ and $\mathring{N}$ the inward normal unit vector to $S_{p,\rh}(w)$; by the Weingarten equation  $\mathring{h}_{ij}= - g( \nabla _{Z_i} \mathring{N}, Z_j)$.

\begin{lem} \label{lem2forma}
The second fundamental form on $S_{p,\rh}(w)$ has the following expansion:
\begin{eqnarray}
\mathring{h}_{ij}&=& \rho (1-w) g(\Theta _i, \Theta _j) + \rho (Hess_{S^2}w)_{ij}+ \frac{2}{3} R(0i0j)\rho ^3 (1-w)^3+\frac{5}{12} \nabla_0  R(0i0j)\rho ^4 (1-w)^4 \nonumber \\
&& + \Big[\frac{3}{20}\nabla_{00} R_p(0i0j)+ \frac{2}{15} R(0i0\mu)R(0j0\mu)\Big] \rho ^5 (1-w)^5-\rho B^{k}_{ij}w_k + O_p(\rho ^6)+\rho ^5 L_p(w)+\rho^2 Q_p^{(2)}(w) \nonumber
\end{eqnarray} 
where  $B^{k}_{ij}$ are functions on $S^2$ of the form $B^k_{ij}=O(\rho^2)+L_p(w)+Q^{(2)}_p(w)$ and, as usual, all curvature terms and scalar products are evaluated at $p$.
\end{lem}

\begin{proof}
In \cite{PX} the authors consider $\mathring {\tilde{N}}$ such that the normal unit vector $\mathring{N}$ has the form $\mathring {N}=\mathring {\tilde{N}} (1 - \rho^2 \mathring{g}^{ij} w_i w_j)^{-1/2}$. They set 
$$\mathring{\tilde{h}}_{ij} = -g(\nabla_{Z_i}\mathring {\tilde{N}},Z_j )$$
and they derive the following formula  
$$\mathring{\tilde{h}}_{ij} = \frac{1}{2(1-w)}  \partial_\rho {\mathring{g}} - \frac{1}{1-w} \rho dw\otimes dw + \rho Hess_{\mathring{g}}w.$$

Using Lemma \ref{lem1forma} the first summand is:
\begin{eqnarray}
 \frac{1}{2(1-w)} \partial_\rho {\mathring{g}}&=& g(\Theta _i, \Theta _j) \rho (1-w) + \frac{\rho}{1-w} w_i w_j + \frac{2}{3} R(0i0j) \rho ^3 (1-w)^3+\frac{5}{12} \nabla_0 R(0i0j)\rho ^4 (1-w)^4 \nonumber \\
&&+\frac{3}{20}\Big[\nabla_{00} R(0i0j)+\frac{2}{15} R(0i0\mu) R(0j0\mu) \Big] \rho ^5 (1-w)^5 + O_p(\rho ^6)+\rho ^6 L_p(w)+\rho^6 Q_p^{(2)}(w)\nonumber 
\end{eqnarray}

The third summand is:

$$ \rho (Hess_{\mathring{g}}w)_{ij} = \rho (w_{ij}-\mathring{\Gamma}^k_{ij} w_k).$$

With a direct computation it is easy to check that

\begin{equation}
\mathring{\Gamma}^k_{ij} = \Gamma^k_{ij}+ B^k_{ij} 
\end{equation}  

where  $\Gamma^k_{ij}$ are the Christoffel symbols of $S^2$ in polar coordinates and
$ B^k_{ij}$ are functions on $S^2$ of the form 
$$B^k_{ij}=O(\rho^2)+L_p(w)+Q^{(2)}_p(w).$$ 

Hence
$$ \rho (Hess_{\mathring{g}}w)_{ij} =  \rho (Hess_{S^2}w)_{ij}- \rho B^k_{ij} w_k. $$

Observing that the second summand simplifies with an adding of the first summand  and that 
$$\mathring{h}_{ij}=- g( \nabla _{Z_i} \mathring{N}, Z_j)=- g( \nabla _{Z_i}\mathring {\tilde{N}} (1 - \rho^2 \mathring{g}^{ij} w_i w_j)^{-1/2}, Z_j)=\mathring{\tilde{h}}_{ij}+\rho^2 Q_p^{(2)}(w) $$
we get the desired formula.
\end{proof}

Recall that the mean curvature $H$ is the trace of $\mathring{h}$ with respect to the metric $\mathring{g}$: $H=\mathring{h}_{ij}\mathring{g}^{ij}$. Collecting the two previous Lemmas we obtain the following

\begin{lem} \label{lemH}
The mean curvature of the hypersurface $S_{p,\rho}(w)$ can be expanded as 

\begin{eqnarray}
H &=& \frac{2}{\rho}+\frac{1}{\rho}(2+ \triangle _{S^2} ) w  +\frac{1}{\rho} \big[ 2w(w+\triangle_{S^2}w) - g^{ij}_{S^2} w_i w_j\big]-\frac{1}{\rho} g^{ij}_{S^2}B^k_{ij} w_k \nonumber\\
&& - \frac{1}{3} \big[g^{il}_{S^2} R(0l0k) g^{kj}_{S^2} (Hess_{S^2}w)_{ij}+ Ric_p(\Theta, \Theta) (1-w)\big] \rho  + \frac{1}{4} g^{ij}_{S^2} \nabla_0 R(0i0j) \rho ^2 (1-w)^2 \nonumber\\
&&+\Big[\frac{1}{10}g^{ij}_{S^2} \nabla_{00} R(0i0j)+ \frac{4}{45} g^{ij}_{S^2} R(0i0\mu) R(0j0\mu) - \frac{1}{9} g^{il}_{S^2} R(0l0k) g^{kn}_{S^2} R(0n0i) \Big] \rho ^3 (1-w)^3 \nonumber \\
&&+ O_p(\rho ^4)+\rho^2 L_p(w)+ Q_p^{(2)}(w)+ \frac{1}{\rho} L_p(w) Q_p^{(2)}(w) \nonumber 
\end{eqnarray}
where $Ric_p$ is the Ricci tensor computed at $p$.
\end{lem}  

\begin{proof}
First let us find an expansion of $\mathring{g}^{ij}$. Given an invertible matrix $A$,  
$$(A+B\rho^2 +C\rho ^3+D \rho^4)^{-1}=A^{-1}-A^{-1}BA^{-1}\rho ^2-A^{-1}CA^{-1} \rho^3-A^{-1}DA^{-1} \rho ^4 + A^{-1} B A^{-1} BA^{-1} \rho^4 +O(\rho ^5)$$
so we get
\begin{eqnarray}
\mathring{g}^{ij}&=& \frac{1}{\rho ^2 (1-w)^2} \Big\{g^{ij}_{S^2}-g^{il}_{S^2} (1-w)^{-2} w_l w_k  g^{kj}_{S^2}-\frac{1}{3}g^{il}_{S^2}R(0l0k)  g^{kj}_{S^2}\rho ^2 (1-w)^2-\frac{1}{6}g^{il}_{S^2} \nabla_0 R(0l0k)   g^{kj}_{S^2}\rho ^3 (1-w)^3 \nonumber\\
&&-g^{il}_{S^2} \Big[\frac{1}{20} \nabla_{00} R(0l0k)+\frac {2}{45} R(0l0\mu)R(0k0\mu)\Big]  g^{kj}_{S^2}\rho ^4 (1-w)^4 +\frac{1}{9} g^{il}_{S^2} R(0l0k) g^{kn}_{S^2} R(0n0q)  g^{qj}_{S^2}\rho ^4 (1-w)^4 \Big\} \nonumber \\
&&+ O_p(\rho ^3)+\rho ^3 L_p(w)+ \rho^2 Q_p^{(2)}(w)+\frac{1}{\rho^2} (Dw)^4. \nonumber\\
 \label{eq:ginv}
\end{eqnarray}
Where $(Dw)^4$ is an homogeneous polynomial in the first derivatives $w_i$ of order four.  
Putting together \eqref{eq:ginv} and Lemma \ref{lem2forma} it is easy to evaluate $H=\mathring{h}_{ij} \mathring{g}^{ij}$ just using the following observations:

\begin{eqnarray}
\bullet  \rho \mathring{g}^{ij} (Hess_{S^2}w)_{ij}&=& \big( \frac{1}{\rho}  (1+2w) g^{ij}_{S^2} - \frac{\rho}{3}g^{il}_{S^2} R(0l0k) g^{kj}_{S^2} +\frac{1}{\rho}Q(w)+O(\rho^2)+ \rho L(w)\big)(Hess_{S^2}w)_{ij} \nonumber\\
&=& \frac{1}{\rho}  (1+2w) \triangle _{S^2} w-  \frac{\rho}{3}g^{il}_{S^2} R(0l0k) g^{kj}_{S^2} (Hess_{S^2}w)_{ij}+ \rho^2 L(w)+\rho Q(w)+ \frac{1}{\rho} L(w)Q(w) \nonumber
\end{eqnarray}
$\bullet$ with a Taylor expansion
$$\frac{2}{\rho (1-w)}= \frac{2(1+w+w^2)}{\rho}+\frac{1}{\rho}w Q(w) ,$$
$$\frac{1}{\rho (1-w)^{3}} g^{ij}_{S^2} w_i w_j=\frac{1}{\rho} g^{ij}_{S^2} w_i w_j+\frac{1}{\rho}w Q(w)  $$
$\bullet$ finally, recalling our notations, \eqref{eqRic} and that $\{\Theta,\frac{\Theta_1}{\|\Theta_1\|},\frac{\Theta_2}{\|\Theta_2\|}  \} $ form an orthonormal base of $T_pM$ 
\begin{equation}\label{eq:R-Ric}
g(\Theta _i, \Theta _j)  g^{il}_{S^2} R(0l0k) g^{kj}_{S^2}=\delta_j ^l g(R_p(\Theta ,\Theta _l)\Theta ,\Theta _k) g^{kj}_{S^2}= g(R_p(\Theta ,\Theta _i)\Theta ,\Theta _j) g^{ij}_{S^2} = -Ric_p(\Theta,\Theta). 
\end{equation}

\end{proof}

Now we compute  $H^2$:
\begin{lem}\label{lemH2}
The square of the mean curvature $H^2$ on $S_{p,\rho}(w)$ can be expanded as
\begin{eqnarray}
H^2 &=& \frac{4}{\rho ^2}+ \frac{4}{\rho ^2}(2+ \bigtriangleup _{S^2})w + \frac{1}{\rho^2}(12 w^2+12w\bigtriangleup _{S^2}w +(\bigtriangleup _{S^2}w)^2-4  g^{ij}_{S^2} w_i w_j)-\frac{4}{\rho^2} g^{ij}_{S^2} B^k_{ij} w_k \nonumber \\
&& - \frac{4}{3}g^{il}_{S^2} R(0l0k) g^{kj}_{S^2} (Hess_{S^2}w)_{ij} -\frac{2}{3} Ric_p(\Theta, \Theta) (2+\triangle_{S^2}w) + [ g^{ij}_{S^2} \nabla _0 R(0i0j)] \rho  \nonumber \\
&&+\Big[ \frac{2}{5}g^{ij}_{S^2} \nabla_{00}R(0i0j)+ \frac{16}{45} g^{ij}_{S^2} R(0i0\mu) R(0j0\mu)- \frac{4}{9} g^{il}_{S^2} R(0l0k) g^{kn}_{S^2} R(0n0i)+ \frac{1}{9} Ric_p(\Theta, \Theta) Ric_p(\Theta, \Theta) \Big] \rho ^2  \nonumber \\
&&+ O_p(\rho ^3)+\rho L_p(w)+ \frac{1}{\rho} Q_p^{(2)}(w)+ \frac{1}{\rho^2} L_p(w) Q_p^{(2)}(w). \nonumber 
\end{eqnarray}
\end{lem}

\begin{proof} Just compute the square of $H$ expressed as in Lemma \ref{lemH}. 
\end{proof}

\begin{lem}\label{lem:detg}
The determinant of the first fundamental form of $S_{p,\rho}(w)$ can be expanded as 
\begin{eqnarray}
det[\mathring{g}]&=& \| \Theta _2 \| ^2 \rho^4 \Big\{ (1-w)^4 + (g_{S^2}^{ij} w_i w_j)-\frac {1}{3} Ric_p (\Theta ,\Theta ) \rho ^2 (1-w)^6 
 + \frac{1}{6} g^{ij}_{S^2} \nabla _0 R(0i0j) \rho ^3 (1-w)^7 \nonumber \\
&&+ \Big[ \frac{1}{20} g^{ij}_{S^2} \nabla _{00} R(0i0j)+\frac {2}{45} g^{ij}_{S^2} R(0i0\mu)R(0j0\mu)+ \frac {1}{9} R(0101)R(0\bar{2}0\bar{2}) -R(010\bar{2})^2\Big]\rho ^4 (1-w)^8 \Big \} \nonumber \\
&&  +O_p(\rho ^9)+\rho ^9 L_p(w)+\rho ^6 Q_p^{(2)}(w)+\rho^4 L_p(w) Q_p^{(2)}(w) \nonumber  
\end{eqnarray}
where recall that $R(0101)= g(R_p(\Theta ,\Theta _1)\Theta ,\Theta _1)$, $R(0\bar{2}0\bar{2})= g(R_p(\Theta ,\bar{\Theta} _2)\Theta ,\bar{\Theta} _2)$, 
$R(010\bar{2})=-g(R_p(\Theta ,\Theta _1)\Theta ,\bar{\Theta} _2)$ and $\bar{\Theta} _2$ is $\Theta _2$ normalized: $\bar{\Theta} _2:=\frac{\Theta _2}{|\Theta _2|} $.
\end{lem}
\begin{proof}
Just compute $det[\mathring{g}]$ using Lemma \ref{lem1forma}, formula \eqref{eq:R-Ric} and observing that $g^{ij}_{S^2}=diag(1,1/\|\Theta_2\|^2)$
\end{proof}

\begin{lem}\label{lem:deth}
The determinant of the second fundamental form of $S_{p,\rho}(w)$ has the following expansion:
\begin{eqnarray}
det[\mathring{h}]&=&  \rho ^2 (1-w)^2 \| \Theta _2\| ^2 + \rho ^2 \| \Theta _2\| ^2 \triangle _{S^2}w (1-w)+\rho ^2 \big[(Hess_{S^2}w)_{11} (Hess_{S^2}w)_{22} - (Hess_{S^2}w)_{12} ^2 \big] \nonumber \\ 
&&+ \frac{2}{3} \rho^4 \big[R(0101) (Hess_{S^2}w)_{22} + R(0202)(Hess_{S^2}w)_{11}-2 R(0102)(Hess_{S^2}w)_{12}-Ric_p(\Theta , \Theta )(1-w)^4  \| \Theta _2\| ^2 \big]\nonumber \\
&&+ \| \Theta _2 \|^ 2\Big[\frac{5}{12} g^{ij}_{S^2} \nabla _0 R(0i0j) \rho ^5 (1-w)^5 +\frac{3}{20} g^{ij}_{S^2} \nabla _{00}R(0i0j) \rho ^6 (1-w)^6+ \frac{2}{15} g^{ij}_{S^2} R(0i0\mu) R(0j0\mu) \rho ^6 (1-w)^6 \Big]\nonumber \\
&&+ \frac{4}{9}\rho ^6 (1-w)^6 \| \Theta _2 \| ^2 \big[R(0101)R(0\bar{2}0\bar{2})-R(010\bar{2})^2\big]-\|\Theta_2\|^2 \rho^2 g^{ij}_{S^2} B_{ij}^k w_k+ O_p(\rho ^7)+\rho ^5 L_p(w)+\rho ^3 Q_p^{(2)}(w).\nonumber
\end{eqnarray} 
\end{lem}

\begin{proof} Just compute the determinant of $\mathring{h}_{ij}$ expressed as in Lemma \ref{lem2forma} using the same tricks of the previous Lemmas.
\end{proof}

\begin{lem}\label{lem:D}
The product of the principal curvatures of $S_{p,\rho}(w)$ 
 
$$D=k_1k_2=\frac {det(\mathring {h})}{det( \mathring {g})}$$

has the following expansion:
\begin{eqnarray}
D&=& \frac{1}{ \rho ^2}(1+2w+ \triangle _{S^2}w+3w\triangle _{S^2}w +3w^2 ) - \frac{1}{\rho^2}g^{ij}_{S^2}w_iw_j +\frac{1}{\|\Theta_2\|^2 \rho ^2}   \big[(Hess_{S^2}w)_{11} (Hess_{S^2}w)_{22} - (Hess_{S^2}w)_{12} ^2 \big] \nonumber \\
&&+ \frac{2}{3\|\Theta_2\|^2 } \Big[ R(0101) (Hess_{S^2}w)_{22} + R(0202) (Hess_{S^2}w)_{11}-2 R(0102) (Hess_{S^2}w)_{12} \Big] \nonumber\\
&&+\frac{1}{3} Ric_p(\Theta,\Theta) (\triangle_{S^2}w-1) + \frac{1}{4} g^{ij}_{S^2} \nabla _0 R(0i0j) \rho (1-w)-\frac{1}{\rho^2} g^{ij}_{S^2} B^k_{ij}w_k \nonumber \\ 
&&  + \Big[\frac{1}{10} g^{ij}_{S^2} \nabla _{00} R(0i0j) +\frac{4}{45} g^{ij}_{S^2} R(0i0\mu) R(0j0\mu)+ \frac{1}{3} [R(0101) R(0\bar{2}0\bar{2})-R(010\bar{2})^2] -\frac {1}{9} Ric_p (\Theta ,\Theta )^2 \Big]\rho ^2 (1-w)^2 \nonumber \\
&&+ O_p(\rho ^3)+\rho L_p(w)+\frac{1}{\rho} Q_p^{(2)}(w) +\frac{1}{\rho^2}L_p(w)  Q_p^{(2)}(w).\nonumber
\end{eqnarray}

\end{lem}

\begin{proof}
Recalling the expansion $\frac{1}{1+x}=1-x+x^2+O(x^3)$ and Lemma \ref{lem:detg} we get
\begin{eqnarray}
\frac{1}{det[\mathring{g}]}&=&\frac {1}{\| \Theta _2 \| ^2 (1-w)^4 \rho ^4} \quad \{1- (g_{S^2}^{ij} w_i w_j)+\frac {1}{3} Ric_p (\Theta ,\Theta ) \rho ^2 (1-w)^2 - \frac{1}{6} g^{ij}_{S^2} \nabla _0 R(0i0j) \rho ^3 (1-w)3 \nonumber \\
&&- \Big[ \frac{1}{20} g^{ij}_{S^2} \nabla _{00} R(0i0j)+\frac {2}{45} g^{ij}_{S^2} R(0i0\mu)R(0j0\mu)+ \frac {1}{9} R(0101)R(0\bar{2}0\bar{2}) -R(010\bar{2})^2\Big]\rho ^4 (1-w)^8 \nonumber\\
&& +\frac {1}{9} Ric_p (\Theta ,\Theta )^2 \rho ^4 (1-w)^4  +O_p(\rho ^5)+\rho ^5 L_p(w)+\rho ^2 Q_p^{(2)}(w)+L_p(w) Q_p^{(2)}(w)\}.  \nonumber
\end{eqnarray}
Gathering together this formula and the expansion of $det(\mathring{h})$ of Lemma \ref{lem:deth} we can conclude.
\end{proof}

The quantity we have to integrate is $\frac{H^2}{4}-D$; collecting the previous Lemmas we finally get the following 

\begin{pro}\label{lem:H2-D}
The integrand of the conformal Willmore functional has the following expansion:
\begin{eqnarray}
\frac{H^2}{4}-D&=&\frac{1}{ \rho ^2}\Big[\frac{1}{4} (\bigtriangleup _{S^2}w)^2-\frac{1}{\|\Theta_2\|^2} (Hess_{S^2}w)_{11} (Hess_{S^2}w)_{22} + \frac{1}{\|\Theta_2\|^2} (Hess_{S^2}w)_{12} ^2\Big] \nonumber \\ 
&&+  \frac{1}{3\|\Theta_2\|^2} \Big[2 R(0102)(Hess_{S^2}w)_{12}-R(0101)(Hess_{S^2}w)_{22} -R(0\bar{2}0\bar{2}) (Hess_{S^2}w)_{11} \Big]\nonumber \\ 
&&+\frac{1}{9} \rho^2 \Big[ \frac{1}{4} Ric_p(\Theta, \Theta )^2 -R(0101)R(0\bar{2}0\bar{2})+R(010\bar{2})^2 \Big]-\frac{1}{6} Ric_p(\Theta, \Theta)\bigtriangleup _{S^2}w    \nonumber \\
&& + O_p(\rho ^3)+\rho L_p(w)+ \frac{1}{\rho} Q_p^{(2)}(w)+ \frac{1}{\rho^2}L_p(w) Q_p^{(2)}(w) \nonumber
\end{eqnarray}

\end{pro}

\begin{proof}
Putting together the formulas of Lemma \ref{lemH2} and Lemma \ref{lem:D}, we get 
\begin{eqnarray}
\frac{H^2}{4}-D &=&  \frac{1}{4\rho ^2}(\bigtriangleup _{S^2}w)^2 +\frac{1}{\|\Theta_2\|^2 \rho ^2} \big[(Hess_{S^2}w)_{12} ^2 -(Hess_{S^2}w)_{11} (Hess_{S^2}w)_{22}   \big]\nonumber\\
&&- \frac{1}{3}g^{il}_{S^2} R(0l0k)  g^{kj}_{S^m} (Hess_{S^2}w)_{ij}- \frac{1}{2} Ric_p(\Theta, \Theta)\bigtriangleup _{S^2}w\nonumber \\ 
&&+  \frac{2}{3\|\Theta_2\|^2} \big[2 R(0102)(Hess_{S^2}w)_{12}- R(0101)(Hess_{S^2}w)_{22}- R(0202)(Hess_{S^2}w)_{11}\big] \nonumber \\
&&- \frac{1}{9}  \rho ^2 \big[ g^{ik}_{S^2} g^{jl}_{S^2} R(0i0l) R(0j0k)\big] -\frac{1}{3} \rho^2 \big[R(0101) R(0\bar{2}0\bar{2})-R(010\bar{2})^2 \big]    \nonumber \\
&&+ \frac{5}{36} Ric_p(\Theta, \Theta )^2 \rho ^2+ O_p(\rho ^3)+\rho L_p(w)+ \frac{1}{\rho} Q_p^{(2)}(w)+ \frac{1}{\rho^2} L_p(w) Q_p^{(2)}(w)\nonumber 
\end{eqnarray}

Let us simplify the second and the third lines; they can be rewritten as
\begin{eqnarray}
&& - \frac{1}{3} R(0101) (Hess_{S^2}w)_{11}- \frac{1}{3} R(0\bar{2}0\bar{2}) \frac{1}{\|\Theta_2\|^2} (Hess_{S^2}w)_{22}+ \frac{2}{3\|\Theta_2\|^2} R(0102) (Hess_{S^2}w)_{12}\nonumber \\
&&- \frac{2}{3} R(0101)\frac{1}{\|\Theta_2\|^2 }(Hess_{S^2}w)_{22} -\frac{2}{3} R(0\bar{2}0\bar{2})(Hess_{S^2}w)_{11}- \frac{1}{2} Ric_p(\Theta, \Theta)\bigtriangleup_{S^2}w\nonumber \\ 
&=&  \frac{2}{3\|\Theta_2\|^2} R(0102)(Hess_{S^2}w)_{12} + \frac{1}{3} Ric_p(\Theta, \Theta)  \bigtriangleup _{S^2}w \nonumber \\
&&- \frac{1}{3} R(0101) \frac{1}{\|\Theta_2\|^2}(Hess_{S^2}w)_{22} -\frac{1}{3}R(0\bar{2}0\bar{2})(Hess_{S^2}w)_{11}- \frac{1}{2} Ric_p(\Theta, \Theta)\bigtriangleup _{S^2}w\nonumber \\ 
&=&-\frac{1}{6} Ric_p(\Theta, \Theta)\bigtriangleup _{S^2}w+  \frac{1}{3\|\Theta_2\|^2} \big[2R(0102)(Hess_{S^2}w)_{12}-R(0101)(Hess_{S^2}w)_{22} -R(0\bar{2}0\bar{2})(Hess_{S^2}w)_{11}\big] \nonumber 
\end{eqnarray}

Finally we have to simplify the forth and the fifth lines; they can be rewritten as 

\begin{eqnarray}
&& \Big\{- \frac{1}{9} R(0101)^2- \frac{1}{9}R(0\bar{2}0\bar{2})^2- \frac{2}{9}R(010\bar{2})^2  + \frac{5}{36} Ric_p(\Theta, \Theta )^2  - \frac{1}{3} R(0101)R(0\bar{2}0\bar{2}) +\frac{1}{3}R(010\bar{2})^2\Big\} \rho ^2 \nonumber \\
&=&\Big\{- \frac{1}{9} [R(0101)+R(0\bar{2}0\bar{2})]^2 - \frac{1}{9} R(0101)R(0\bar{2}0\bar{2})+ \frac{1}{9}R(010\bar{2})^2+ \frac{5}{36} Ric_p(\Theta, \Theta )^2\Big\}\rho^2 \nonumber\\
&=& \Big\{ \frac{1}{36}Ric_p(\Theta, \Theta )^2 - \frac{1}{9} R(0101)R(0\bar{2}0\bar{2})+ \frac{1}{9}R(010\bar{2})^2 \Big\}\rho^2 \nonumber
\end{eqnarray}
where, in the last equality, we used the usual identity $R(0101)+R(0\bar{2}0\bar{2})=-Ric_p(\Theta,\Theta)$. 

Collecting the formulas we get the desired expansion.
\end{proof}

\subsection {The differential of the conformal Willmore functional on perturbed geodesic spheres $S_{p,\rh}(w)$}\label{subsec:I'Mg}

\begin{pro}\label{lem:svilI'}
On the perturbed geodesic sphere $S_{p,\rho}(w)$ the differential of the conformal Willmore functional has the following form: 
\begin{equation} \nonumber
I'(S_{p,\rho}(w))=\frac{1}{2\rho^3}\triangle_{S^2}(\triangle _{S^2}+2 ) w -\frac{1}{6\rho} \triangle _{S^2} Ric_p(\Theta, \Theta) + O_p(\rho ^0)+\frac{1}{\rho^2}L_p^{(4)}(w) + \frac{1}{\rho^3}Q_p^{(2)(4)}(w)
\end{equation}
\end{pro}

\begin{proof}
Let us recall the general expression of the differential of the conformal Willmore functional computed in \cite{HL} (Theorem 3.1 plus an easy computation using Codazzi equation).

Given a compact Riemannian surface $(\mathring{M},\mathring{g})$  isometrically immersed in the three dimensional Riemannian manifold $(M,g)$ and called $\mathring{N}$ the inward normal unit vector, the differential of the conformal Willmore functional 
$$ I(\mathring{M})=\int_{\mathring{M}}\bigg(\frac{H^2}{4}-D\bigg) d\Sigma $$
is
$$I'(\mathring{M})=\frac{1}{2}\triangle_{\mathring{M}}H+H\bigg(\frac{H^2}{4}-D\bigg)+\sum_{ij}R(\mathring{N},e_i,\mathring{N},e_j)  \mathring{h}_{ij}-\frac{1}{2}\sum_i HR(\mathring{N},e_i,\mathring{N},e_i)+\sum_{ij}(\nabla_{e_i}R)(\mathring{N},e_j,e_j,e_i)$$
where $e_1,e_2$ is a local orthonormal frame of $T_p\mathring{M}$ which diagonalizes the second fundamental form $\mathring{h}_{ij}$. 
\\Since $e_1,e_2$ are principal directions we get
$$\sum_{ij}R(\mathring{N},e_i,\mathring{N},e_j) \mathring{h}_{ij}-\frac{1}{2}\sum_i H R(\mathring{N},e_i,\mathring{N},e_i)=\frac{(\lambda_1-\lambda_2)}{2} [R(\mathring{N},e_1,\mathring{N},e_1)-R(\mathring{N},e_2,\mathring{N},e_2)]$$
where $\lambda_1,\lambda_2$ are the principal curvatures.
So in this frame the differential is
\begin{equation}\label{eq:diffI}
I'(\mathring{M})= \frac{1}{2}\triangle_{\mathring{M}}H+ H\bigg(\frac{H^2}{4}-D\bigg)+\frac{(\lambda_1-\lambda_2)}{2}[R(\mathring{N},e_1,\mathring{N},e_1)-R(\mathring{N},e_2,\mathring{N},e_2)] +\sum_{ij}(\nabla_{e_i}R)(\mathring{N},e_j,e_j,e_i).
 \end{equation}

Now we want to compute the differential on the perturbed geodesic sphere $S_{p,\rho}(w)$. 
\\Recall that
\begin{eqnarray}
\triangle_{\mathring{g}}u&=&\mathring{g}^{ij}\big(u_{ij}-\mathring{\Gamma}_{ij}^k u_k \big)\nonumber\\
&=& \frac{1}{\rho^2} g_{S^2}^{ij}\big( u_{ij}-\Gamma_{ij}^k u_k \big)+O(\rho^0) L(u)+\frac{1}{\rho^2}L_p(w) L(u)+\frac{1}{\rho^2}Q_p^{(2)}(w)L(u)\nonumber\\
&=& \frac{1}{\rho^2} \bigtriangleup _{S^2}u+O(\rho^0) L(u)+\frac{1}{\rho^2}L_p(w)L(u)+\frac{1}{\rho^2}Q_p^{(2)}(w)L(u)\nonumber
\end{eqnarray}
where $L(u)$ is a linear function depending on $u$ and on its derivatives up to order two.
From the above computation of $H$  we have
$$H = \frac{2}{\rho}+\frac{1}{\rho}(2+ \triangle _{S^2} ) w -\frac{1}{3} Ric_p(\Theta, \Theta)\rho+O(\rho^2)+\rho L_p(w)+\frac{1}{\rho}Q^{(2)}_p(w),$$
hence
\begin{eqnarray}
\triangle_{\mathring{g}}H &=&\frac{1}{\rho^3} \triangle _{S^2}(\triangle _{S^2}+2 ) w -\frac{1}{3\rho} \triangle _{S^2} Ric_p(\Theta, \Theta)+ O_p(\rho ^0)+\frac{1}{\rho^2}L_p^{(4)}(w)+\frac{1}{\rho^3}Q_p^{(2)(4)}(w) . \nonumber 
\end{eqnarray} 

Now let us show that the other summands are negligible.
\\First we find an expansion for the principal directions $\lambda_1$ and $\lambda_2$. From the definitions, they are the roots of the polynomial equation
$$x^2-Hx+D=0$$
so 
$$\lambda_{1,2}=\frac{H}{2}\pm \frac{\sqrt{H^2-4D}}{2}=\frac{1}{\rho}+O(\rho)+\frac{1}{\rho}L_p (w)+\frac{1}{\rho}Q^{(2)}_p(w)$$
and the third summand is negligible:
$$(\lambda_1-\lambda_2)[R(\mathring{N},e_1,\mathring{N},e_1)-R(\mathring{N},e_2,\mathring{N},e_2)]=O(\rho)+\frac{1}{\rho}L_p(w)+\frac{1}{\rho}Q^{(2)}_p(w).$$
From the above computation of  $\frac{H^2}{4}-D$, we have
$$\frac{H^2}{4}-D=O_p(\rho^2)+L_p(w)+\frac{1}{\rho^2}Q^{(2)}_p(w)$$
hence we get
$$H\bigg( \frac{H^2}{4}-D \bigg)= O_p(\rho )+\frac{1}{\rho} L_p(w)+  \frac {1}{\rho^3} Q_p^{(2)}(w).$$
Therefore also this term is negligible and we can conclude observing that $(\nabla_{e_i}R)(\mathring{N},e_j,e_j,e_i)=O(\rho^0)$.

\end{proof}

\subsection{The finite dimensional reduction}\label{subsec:FDRnon}
NOTATION. In this subsection, the functional space will be $C^{4,\a}(S^2)^\perp$: the perturbation $w$ will be an element of $C^{4,\a}(S^2)^\perp$ and $B(0,r)$ will denote the ball of center $0$ and radius $r$ in $C^{4,\a}(S^2)^\perp$. 

\begin{lem}\label{lemEqAuxContr}
Fixed a compact subset $Z_c \subseteq M,$ there exist $\rho_0>0$, $r>0$ and a map $w_{(.,.)}:Z_c \times [0,\rh_0] \to C^{4,\alpha}(S^2)^\perp$, $(p,\rh) \mapsto w_{p,\rh}$ such that if $S_{p,\rh}(w)$ is a critical point of the conformal Willmore functional $I$ with $(p,\rh,w) \in Z_c\times [0,\rh_0] \times B(0,r)$ then $w=w_{p,\rh}$. 
\\Moreover the map $w_{(.,.)}$ satisfies the following properties:
\\$(i) $ the map $(p,\rho)\mapsto w_{p,\rho}$ is $C^1$,
\\$(ii)$  $\| w_{p,\rho}\|_{C^{4,\alpha}(S^2)}=O(\rho^2)$ as $\rho\to 0$ uniformly for $p\in Z_c$,
\\$(iii)$  $\| \frac{\partial}{\partial \rh }w_{p,\rho}\|_{L^2(S^2)}=O(\rho)$ as $\rho\to 0$ uniformly for $p\in Z_c$,
\\$(iv)$ we have the following explicit expansion of $w_{p,\rho}$:
\begin{equation}\label{eq:svilw}
w_{p,\rho}= -\frac{1}{12}\rho^2 Ric_p(\Theta,\Theta)+\frac{1}{36}\rho^ 2 R(p)+ O(\rho^3)
\end{equation} 
where the remainder $O(\rh^3)$ has to be intended in $C^{4,\a}(S^2)$ norm.
\end{lem}

\begin{proof}
For the proof of $(i),(ii)$ and $(iii)$ we refer to Lemma 4.4 of \cite{Mon}, here we only give a sketch of the idea.
Fixed  a compact subset $Z_c \subseteq M$ and $p\in Z_c$, if
$$I'(S_{p,\rho}(w))=0\quad (\text{equality in } L^2(S^2) ),$$
then, setting $P: L^2(S^2) \to Ker[ \triangle _{S^2}(\triangle _{S^2} + 2 )]^\perp $ the orthogonal projection, a fortiori we  have
$$PI'(S_{p,\rho}(w))=0;$$
that is, using the expansion of Proposition \ref{lem:svilI'},
\begin{equation} \label{eqdiffw}
P \Big[  \triangle _{S^2}(\triangle _{S^2} + 2 ) w +O_p(\rho^2) +\rho L_p^{(4)}(w)+ Q_p^{(2)(4)}(w) \Big] =0.
\end{equation}
Since  $ \triangle _{S^2}(\triangle _{S^2} + 2 )$ is invertible on the space orthogonal to the Kernel and $w \in C^{4,\alpha}(S^2)^\perp:=Ker[ \triangle _{S^2}(\triangle _{S^2} + 2 )]^\perp \cap C^{4,\alpha}(S^2)$, setting 
$$K:= [\triangle _{S^2}(\triangle _{S^2} + 2 )]^{-1}: Ker[ \triangle _{S^2}(\triangle _{S^2} + 2 )]^\perp\subseteq L^2(S^2) \to  Ker[ \triangle _{S^2}(\triangle _{S^2} + 2 )]^\perp,$$ 
the equation \eqref{eqdiffw} is equivalent to the fixed point problem
\begin{equation} \label{puntofissow}
 w = K [O_p(\rho^2) +\rho L_p^{(4)}(w)+ Q_p^{(2)(4)}(w)]=F_{p,\rho} (w).
\end{equation}
The projection in the right hand side is intended.
In the aforementioned paper (using Schauder estimates) it is proved that once  the compact $Z_c \subset M$ is fixed, there exist $\rho_0>0$ and $r>0$ such that for all $p \in Z_c$ and $\rho<\rho_0$ the map
$$F_{p,\rho}:B(0,r)\subset C^{4,\a}(S^2)^\perp \to C^{4,\a}(S^2)^\perp$$
is a contraction. In the same paper the regularity and the decay properties are shown.
\\

Now let us prove the expansion $(iv)$. 
\\Using the formula of Proposition \ref{lem:svilI'}, the unique solution $w \in B(0,r)$ to the fixed point problem will have to satisfy the following fourth order elliptic PDE: 
\begin{equation} \nonumber
\triangle_{S^2}(\triangle _{S^2}+2 ) w =\frac{1}{3} \rho^2 \triangle _{S^2} Ric_p(\Theta, \Theta) + O_p(\rho^3)+\rho L_p^{(4)}(w) +Q_p^{(2)(4)}(w).
\end{equation}
Clearly the unique solution $w$ has the form $w=\rh^2 \bar{w}+O(\rh ^3)$ where the remainder has to be intended in $C^{4,\a}(S^2)$ norm and $\bar{w} \in C^{4,\a}(S^2)$ is independent on $\rh$. Now we want to find an explicit formula for $\bar{w}$.
Writing the radial unit vector in normal coordinates on $T_pM$, we have $\Theta=x^i \frac{\partial}{\partial x^i}$ and the Ricci tensor can be written as
$$Ric_p(\Theta, \Theta)= \sum_{i\neq j} R_{ij}x^ix^j+ \sum_i R_{ii}(x^i)^2.$$
Recall that the eigenfunctions of $\triangle_{S^2}$ relative to the second eigenvalue $\lambda_2=-6$ are $x^i x^j, i\neq j$ and 
$(x^i)^2-(x^j)^2, i\neq j$ and notice that
$$2(x^1)^2-1=(x^1)^2-(x^2)^2+(x^1)^2-(x^3)^2-(x^1)^2$$
so
$$ (x^1)^2 - \frac{1}{3}=\frac{1}{3}\bigg\{[(x^1)^2-(x^2)^2]+[(x^1)^2-(x^3)^2]\bigg\} $$
is an element of the eigenspace relative to $\lambda_2=-6$ (analogously for the others $(x^i)^2$).
So 
\begin{eqnarray}
Ric_p(\Theta,\Theta)&=&\sum_{i\neq j}R_{ij}(p)x^ix^j+\sum_{i}R_{ii}(p)[(x^i)^2-\frac{1}{3}]+\frac{1}{3}\sum_{i}R_{ii}(p)\nonumber\\
&=&\sum_{i\neq j}R_{ij}(p)x^ix^j+\sum_{i}R_{ii}(p)[(x^i)^2-\frac{1}{3}]+\frac{1}{3}R(p) \nonumber
\end{eqnarray}
and $Ric_p(\Theta,\Theta)-\frac{1}{3}R(p) $ is an element of the second eigenspace of $\triangle_{S^2}$.
\\Recalling that $w=\rho^2 \bar{w}+O(\rho^3)$, then $\bar{w}$ has to solve the following linear elliptic PDE
\begin{equation} \nonumber
\triangle_{S^2}(\triangle _{S^2}+2 ) \bar{w} =\frac{1}{3}  \triangle _{S^2} \Big[Ric_p(\Theta, \Theta)-\frac{1}{3}R(p) \Big]. 
\end{equation} 
Since the right hand side is an eigenfunction of  $\triangle _{S^2}$ with eigenvalue $-6$ the equation is easily solved as
$$\bar{w}=-\frac{1}{12}Ric(\Theta,\Theta)+\frac{1}{36}R(p).$$
\end{proof}

\subsection{The expansion of the reduced functional $I(S_{p,\rh}(w_{p,\rh}))$}\label{subsec:ExpIMg}
In this subsection we want to evaluate the reduced functional $I(S_{p,\rho}(w_{p,\rho}))$, that is the conformal Willmore functional on perturbed geodesic spheres with perturbation $w$ in the constraint given by Proposition \ref{lemEqAuxContr}.  
\begin{pro}\label{svilPhipr}
The conformal Willmore functional on perturbed geodesic spheres $S_{p,\rho}(w_{p,\rho})$ with perturbation $w_{p,\rho}$ lying in the constraint given by Proposition \ref{lemEqAuxContr} can be expanded in $\rho$ as follows
$$I(S_{p,\rho}(w_{p,\rho}))=  \frac{\pi}{5} \|S_p\|^2 \rho^4+ O_p(\rho ^5),$$
where $S_p$ is the Traceless Ricci tensor defined in \eqref{def:S}. 
\end{pro}

\begin{proof}
In the sequel we fix a point $p\in M$ and we want to evaluate $I(S_{p,\rh}(w_{p,\rh}))$ for small $\rho$. For simplicity of notation, let us denote $w=w_{p,\rho}$; from Proposition \ref{lemEqAuxContr} we know that $w=\rho^2 \bar{w}+O(\rh^3)$. Notice that the leading part of $H^2/4-D$ is homogeneous of  degree two in $\rho$, so in order to evaluate $I(S_{p,\rho}(w_{p,\rho}))$ it is sufficient to multiply $H^2/4-D$ by the first term of $\sqrt{det[\mathring{g}]}$ (that is $\rho^2 \|\Theta_2\|$). Using the expansion of Proposition \ref{lem:H2-D} we get

\begin{eqnarray}
I(S_{p,\rho}(w))&=&\rho^4 \int_{S^2} \bigg[ \frac{1}{4}(\bigtriangleup _{S^2}\bar{w})^2-\frac{1}{\|\Theta_2\|^2}  (Hess_{S^2}\bar{w})_{11} (Hess_{S^2}\bar{w})_{22} + \frac{1}{\|\Theta_2\|^2} (Hess_{S^2}\bar{w})_{12} ^2 -\frac{1}{6}  Ric_p(\Theta, \Theta)\bigtriangleup _{S^2}\bar{w}\nonumber \\ 
&&+  \frac{2}{3\|\Theta_2\|^2} R(0102)(Hess_{S^2}\bar{w})_{12}- \frac{1}{3\|\Theta_2\|^2} R(0101) (Hess_{S^2}\bar{w})_{22} -\frac{1}{3} R(0\bar{2}0\bar{2}) (Hess_{S^2}\bar{w})_{11} \nonumber \\ 
&&+\frac{1}{9}  \Big( \frac{1}{4} Ric_p(\Theta, \Theta )^2 - R(0101)R(0\bar{2}0\bar{2})+ R(010\bar{2})^2  \Big) \bigg]d\Sigma_0 \quad + O_p(\rho ^5). \label{integrale}
\end{eqnarray}
From (iv) of Proposition \ref{lemEqAuxContr} it follows that 
$$\triangle _{S^2}\bar{w}=-6\bar{w}=\frac{1}{2}Ric_p(\Theta,\Theta)-\frac{1}{6}R(p)$$
so, after some easy computation, one can write 
\begin{equation}\label{integrando1}
\frac{1}{4}(\bigtriangleup _{S^2}\bar{w})^2-\frac{1}{6} Ric_p(\Theta, \Theta) \bigtriangleup _{S^2}\bar{w}+\frac{1}{36} Ric_p(\Theta, \Theta)^2 = \frac{1}{144}Ric_p(\Theta, \Theta)^2 -\frac{1}{72} Ric_p(\Theta, \Theta) R(p)+\frac{1}{144} R(p)^2.  
\end{equation}
In order to simplify the other integrands of \eqref{integrale} we compute  $(Hess_{S^2}\bar{w})_{ij}$. The nonvanishing Christoffel symbols of $S^2$ in polar coordinates $\theta^1, \theta^2$ are
\begin{eqnarray}
\Gamma_{12}^2=\Gamma_{21}^2= cotan \theta^1 \nonumber\\
\Gamma_{22}^1=-\sin \theta^1 \cos \theta^1. \nonumber
\end{eqnarray}
Hence, recalling that $(Hess w)_{ij}=w_{ij}-\Gamma^k_{ij}w_k$ and the expression of $w$ given in (iv), we get
\begin{eqnarray}
(Hess_{S^2}\bar{w})_{11}&=&\bar{w}_{11}=-\frac{1}{6}\partial_{\theta^1}( Ric_p(\Theta,\Theta_1))\nonumber\\
&=&-\frac{1}{6} Ric_p(\Theta_1,\Theta_1)-\frac{1}{6} Ric_p(\Theta,\Theta_{11}) \quad \text{but } \Theta_{11}=-\Theta\nonumber\\
&=&-\frac{1}{6} Ric_p(\Theta_1,\Theta_1)+\frac{1}{6} Ric_p(\Theta,\Theta)\nonumber\\
(Hess_{S^2}\bar{w})_{12}&=&\bar{w}_{12}-\Gamma_{12}^2 \bar{w}_2=-\frac{1}{6}\partial_{\theta^1}( Ric_p(\Theta,\Theta_2))+\frac{1}{6}\Gamma_{12}^2 Ric_p(\Theta,\Theta_2)\nonumber\\
&=&-\frac{1}{6} Ric_p(\Theta_1,\Theta_2)-\frac{1}{6}Ric_p(\Theta,\Theta_{12})+\frac{1}{6}\Gamma_{12}^2 Ric_p(\Theta,\Theta_2)    \text{ but } \Theta_{12}=cotan \theta^1 \Theta_2    \nonumber\\
&=&-\frac{1}{6} Ric_p(\Theta_1,\Theta_2)  \nonumber
\end{eqnarray}
\begin{eqnarray}
(Hess_{S^2}\bar{w})_{22}&=&\bar{w}_{22}-\Gamma_{22}^1 \bar{w}_1\nonumber \\
&=&-\frac{1}{6} Ric_p(\Theta_2,\Theta_2)-\frac{1}{6} Ric_p(\Theta,\Theta_{22})+\frac{1}{6}\Gamma_{22}^1 Ric_p(\Theta,\Theta_1) \text{ but } \Theta_{22}=-\sin \theta^1 \cos \theta^1 \Theta_1- \sin^2 \theta^1 \Theta \nonumber\\
&=&-\frac{1}{6} Ric_p(\Theta_2,\Theta_2)+\frac{1}{6} \|\Theta_2\|^2 Ric_p(\Theta,\Theta).\nonumber
\end{eqnarray}
Therefore the other part of the integrand can be written as

\begin{eqnarray}
&&-\frac{1}{\|\Theta_2\|^2}  (Hess_{S^2}\bar{w})_{11} (Hess_{S^2}\bar{w})_{22} + \frac{1}{\|\Theta_2\|^2} (Hess_{S^2}\bar{w})_{12} ^2+  \frac{2}{3\|\Theta_2\|^2} R(0102) (Hess_{S^2}\bar{w})_{12} \nonumber \\ 
&&- \frac{1}{3} R(0101) \frac{1}{\|\Theta_2\|^2}(Hess_{S^2}\bar{w})_{22} -\frac{1}{3} R(0\bar{2}0\bar{2})(Hess_{S^2}\bar{w})_{11}-\frac{1}{9} R(0101)R(0\bar{2}0\bar{2})+ \frac{1}{9}R(010\bar{2})^2 \nonumber \\
&=&- \frac{1}{36} Ric(\Theta, \Theta)^2 + \frac{1}{36} Ric(\Theta, \Theta) [Ric(\bar{\Theta}_2, \bar{\Theta}_2)+Ric(\Theta_1, \Theta_1)]- \frac{1}{36}Ric(\bar{\Theta}_2, \bar{\Theta}_2)Ric(\Theta_1, \Theta_1)\nonumber \\
&&+ \frac{1}{36} Ric(\Theta_1, \bar{\Theta}_2)^2- \frac{1}{9} R(010\bar{2})Ric(\Theta_1, \bar{\Theta}_2)+ \frac{1}{18} R(0101)(Ric(\bar{\Theta}_2, \bar{\Theta}_2)- Ric(\Theta, \Theta)) \nonumber\\
&&+\frac{1}{18} R(0\bar{2}0\bar{2})(Ric(\Theta_1, \Theta_1)- Ric(\Theta, \Theta))-\frac{1}{9} R(0101) R(0\bar{2}0\bar{2})+\frac{1}{9}R(010\bar{2})^2. \label{2intprov}
\end{eqnarray}
Using the following three identities (which follow from the orthogonality of $\{\Theta, \Theta_1, \bar{\Theta}_2 \}$, from the definitions and the symmetries of the curvature tensors)

\begin{eqnarray}
-\frac{1}{18}\Big(R(0101)+R(0\bar{2}0\bar{2})\Big) Ric_p(\Theta,\Theta)&=&\frac{1}{18} Ric_p(\Theta,\Theta)^2 \nonumber\\
Ric_p(\Theta_1, \Theta_1)+Ric_p(\bar{\Theta}_2, \bar{\Theta}_2)&=&R(p)-Ric_p(\Theta, \Theta) \label{id:Ric-R} \\
R(010\bar{2})&=&-Ric_p(\Theta_1,\bar{\Theta}_2), \nonumber
\end{eqnarray}
after some easy computations we can say that \eqref{2intprov} equals  

\begin{eqnarray}
&=&\frac{1}{36} Ric(\Theta, \Theta) R(p)- \frac{1}{36}Ric(\bar{\Theta}_2, \bar{\Theta}_2)Ric(\Theta_1, \Theta_1) +\frac{1}{4} Ric(\Theta_1, \bar{\Theta}_2)^2 \nonumber \\
&&+ \frac{1}{18} R(0101) Ric(\bar{\Theta}_2, \bar{\Theta}_2)+ \frac{1}{18} R(0\bar{2}0\bar{2}) Ric(\Theta_1, \Theta_1) -\frac{1}{9} R(0101) R(0\bar{2}0\bar{2}). 
\end{eqnarray}
Let us try to simplify the last line using that $R(0101)+R(\bar{2}1\bar{2}1) = -Ric(\Theta_1,\Theta_1)$ and identity \eqref{id:Ric-R}:
\begin{eqnarray}
&& \frac{1}{18} R(0101) Ric(\bar{\Theta}_2, \bar{\Theta}_2)+\frac{1}{18} R(0\bar{2}0\bar{2}) Ric(\Theta_1, \Theta_1)- \frac{1}{9} R(0101)R(0\bar{2}0\bar{2})\nonumber\\
&=&-\frac{1}{18} \big[Ric_p(\Theta_1,\Theta_1)+R(1\bar{2}1\bar{2})\big]Ric_p(\bar{\Theta}_2, \bar{\Theta}_2)-\frac{1}{18} \big[Ric(\bar{\Theta}_2, \bar{\Theta}_2)+ R(1\bar{2}1\bar{2})\big] Ric_p(\Theta_1, \Theta_1) \nonumber \\
&&-\frac{1}{9}\big[Ric_p(\Theta_1,\Theta_1)+R(1\bar{2}1\bar{2})\big] \big[Ric(\bar{\Theta}_2, \bar{\Theta}_2)+R(1\bar{2}1\bar{2})\big] \nonumber\\
&=&-\frac{2}{9}Ric_p(\Theta_1, \Theta_1)Ric_p(\bar{\Theta}_2, \bar{\Theta}_2)+\frac{1}{6} R(1\bar{2}1\bar{2}) Ric_p(\Theta,\Theta)-\frac{1}{6} R(1\bar{2}1\bar{2})R(p)-\frac{1}{9} R(1\bar{2}1\bar{2})^2. \label{3add}
\end{eqnarray}
Since $\{\Theta,\Theta_1,\bar{\Theta}_2\}$ is an orthonormal base of $T_pM$ we have the following useful identity
\begin{eqnarray}
R(1\bar{2}1\bar{2})&=&R(1\bar{2}1\bar{2})-Ric_p(\Theta,\Theta)+Ric_p(\Theta,\Theta) \nonumber\\
&=&\big[R(1\bar{2}1\bar{2})+R(0\bar{2}0\bar{2})+R(0101)\big]+Ric_p(\Theta,\Theta) \nonumber\\
&=&-\frac{1}{2}\big[Ric_p(\Theta_1,\Theta_1)+Ric_p(\bar{\Theta}_2,\bar{\Theta}_2)+Ric_p(\Theta,\Theta) \big]+Ric_p(\Theta,\Theta) \nonumber\\
&=&-\frac{1}{2} R(p)+Ric_p(\Theta,\Theta). \label{id:R-RRic}
\end{eqnarray}
Plugging the last identity \eqref{id:R-RRic} into formula \eqref{3add}, we get that \eqref{3add} equals
\begin{equation}\label{3add'}
=-\frac{2}{9}Ric_p(\Theta_1, \Theta_1)Ric_p(\bar{\Theta}_2, \bar{\Theta}_2)-\frac{5}{36}Ric_p(\Theta, \Theta) R(p)+\frac{1}{18}Ric_p(\Theta, \Theta)^2+\frac{1}{18}R(p)^2.
\end{equation}
Therefore the last line of \eqref{2intprov} equals \eqref{3add'} and the integrands \eqref{2intprov} become
$$= \frac{1}{18} Ric(\Theta, \Theta)^2-\frac{1}{9} Ric(\Theta, \Theta) R(p)+\frac{1}{18}R(p)^2+\frac{1}{4} Ric(\Theta_1, \bar{\Theta}_2)^2-\frac{1}{4}Ric_p(\Theta_1, \Theta_1)Ric_p(\bar{\Theta}_2, \bar{\Theta}_2);$$
hence the conformal Willmore functional expressed as in \eqref{integrale}, using the last formula and \eqref{integrando1}, becomes
\begin{eqnarray}
I(S_{p,\rho}(w))=\rho^4 \int_{S^2} \bigg[\frac{1}{16}Ric_p(\Theta, \Theta)^2 -\frac{1}{8} Ric_p(\Theta, \Theta) R(p)+\frac{1}{16} R(p)^2+\frac{1}{4} Ric(\Theta_1, \bar{\Theta}_2)^2 -\frac{1}{4}Ric_p(\Theta_1, \Theta_1)Ric_p(\bar{\Theta}_2, \bar{\Theta}_2)    \bigg]d\Sigma_0 \nonumber\\
+ O_p(\rho ^5)\label{integrale'}.
\end{eqnarray}
The integral of the first three summands is well-known (see for example the appendix of \cite{PX}), let us compute the integral of the last two summands.
\begin{claim} 
$$ \int_{S^2} \big[Ric_p(\Theta_1, \bar{\Theta}_2)^2-Ric_p(\Theta_1, \Theta_1)Ric_p(\bar{\Theta}_2, \bar{\Theta}_2) \big]d\Sigma_0=\frac{2\pi}{3} \big(\|Ric_p\|^2-R(p)^2 \big)$$
\end{claim}

\emph{Proof of the Claim:}
\\As before let us denote with $E_\mu, \mu=1,2,3$ an orthonormal base of $T_pM$ and with $x^\mu$ the induced coordinates. Under this notation the radial unit vector is
$$S^2 \ni \Theta = x^\mu E_\mu.$$
Recall that the polar coordinates $0<\theta^1<\pi, \quad 0<\theta^2<2\pi$ have been chosen such that $S^2$ is parametrized as follows 
\begin{displaymath} 
\left\{ \begin{array}{ll}
x^1= \sin \theta ^1 \cos \theta ^2 \\
x^2= \sin \theta ^1 \sin \theta ^2 \\
x^3= \cos \theta ^1. \\
\end{array} \right.
\end{displaymath}
The normalized tangent vectors $\bar{\Theta}_i:=\frac{\Theta_i}{\|\Theta_i \|}$ have coordinates
\begin{eqnarray}
\bar{\Theta}_1&=& \Theta _1=(\cos \theta ^1 \cos \theta ^2,\cos \theta ^1 \sin \theta ^2,-\sin \theta ^1)\nonumber \\ 
&=& \Bigg( \frac{x^1 x^3}{\sqrt{(x^1)^2+(x^2)^2}},\frac{x^2 x^3}{\sqrt{(x^1)^2+(x^2)^2}}, - \sqrt{(x^1)^2+(x^2)^2}\Bigg) \label{eqtheta1} \\
\bar{\Theta} _2&=&  (-\sin \theta ^2 ,\cos \theta ^2,0)\nonumber \\
&=&\Bigg( -\frac{x^2}{\sqrt{(x^1)^2+(x^2)^2}},\frac{x^1}{\sqrt{(x^1)^2+(x^2)^2}},0 \Bigg). \label{eqtheta2}
\end{eqnarray}
Using this expressions for $\bar{\Theta}_i$ we get the following formulas for $Ric_p(\bar{\Theta_i},\bar{\Theta}_j)$:
\begin{eqnarray}
Ric_p(\Theta_1,\Theta_1)&=&R_{11}\frac{(x^1)^2 (x^3)^2}{(x^1)^2+(x^2)^2} +2R_{12} \frac{x^1 x^2 (x^3)^2}{(x^1)^2+(x^2)^2}+R_{22} \frac{(x^2)^2 (x^3)^2}{(x^1)^2+(x^2)^2}-2R_{13}x^1x^3-2R_{23}x^2x^3+R_{33}[(x^1)^2+(x^2)^2]\nonumber \\
Ric_p(\bar{\Theta}_2,\bar{\Theta}_2)&=&R_{11}\frac{(x^2)^2}{(x^1)^2+(x^2)^2}-2R_{12} \frac{x^1 x^2}{(x^1)^2+(x^2)^2}+R_{22} \frac{(x^1)^2}{(x^1)^2+(x^2)^2}\nonumber \\
Ric_p(\Theta_1,\bar{\Theta}_2)&=&-R_{11}\frac{x^1 x^2 x^3}{(x^1)^2+(x^2)^2}-2R_{12}\frac{(x^2)^2x^3}{(x^1)^2+(x^2)^2}+R_{22}\frac{x^1 x^2 x^3}{(x^1)^2+(x^2)^2}+R_{12} x^3+R_{13}x^2-R_{23}x^1.\nonumber
\end{eqnarray}
Notice that the summands which contain a term of the type $(x^i)^{2m+1}$ ($m \in \N$) have vanishing integral on $S^2$; then, calling ``Remainder'' all these summands, we get
\begin{eqnarray}
Ric(\Theta_1, \bar{\Theta}_2)^2&=&(R_{11}^2-2R_{11}R_{22}+R_{22}^2-4R_{12}^2)\frac{(x^1)^2(x^2)^2(x^3)^2}{[(x^1)^2+(x^2)^2]^2}+R_{12}^2 (x^3)^2+R_{13}^2(x^2)^2+R_{23}^2(x^1)^2\nonumber \\
&&+ \text{ Remainder},\nonumber \\
Ric(\Theta_1, \Theta_1) Ric(\bar{\Theta}_2, \bar{\Theta}_2)&=&(R_{11}^2-2R_{11}R_{22}+R_{22}^2-4R_{12}^2)\frac{(x^1)^2(x^2)^2(x^3)^2}{[(x^1)^2+(x^2)^2]^2}+R_{11}R_{22} (x^3)^2+R_{11}R_{33}(x^2)^2+R_{22}R_{33}(x^1)^2\nonumber\\
 &&+ \text{ Remainder}.\nonumber
  \end{eqnarray}
Therefore the integral of the left hand side of the Claim becomes
$$=\int_{S^2} \bigg[  R_{12}^2 (x^3)^2+R_{13}^2(x^2)^2+R_{23}^2(x^1)^2-R_{11}R_{22} (x^3)^2-R_{11}R_{33}(x^2)^2-R_{22}R_{33}(x^1)^2\bigg]d\Sigma_0.$$
Recalling that $\int_{S^2}(x^\mu)^2 d\Sigma_0 =\frac{4\pi}{3}$,we can continue the equalities
\begin{eqnarray}
&=& \frac{4 \pi}{3} \big[R_{12}^2+R_{13}^2+R_{23}^2-R_{11}R_{22}-R_{11}R_{33}-R_{22}R_{33}\big]   \nonumber \\
&=&  \frac{2 \pi}{3} \big[(R_{11}^2+R_{22}^2+R_{33}^2+2R_{12}^2+2R_{13}^2+2R_{23}^2)-(R_{11}^2+R_{22}^2+R_{33}^2+2R_{11}R_{22}+2R_{11}R_{33}+2R_{22}R_{33})\big]\nonumber \\
&=& \frac{2 \pi}{3} \big(\|Ric_p\|^2-R(p)^2 \big) \nonumber 
\end{eqnarray}

Now we are in position to conclude the computation of the integral \eqref{integrale'}.
\\It is known that $\int_{S^2} Ric_p(\Theta,\Theta) d\Sigma_0=\frac{4\pi}{3}R(p)$ and $\int_{S^2}[Ric_p(\Theta,\Theta)]^2 d\Sigma_0=\frac{4\pi}{15} (2\|Ric_p\|^2+R(p)^2)$ (see the appendix of \cite{PX}) thus, grouping together this formulas and the claim,  we can say that the conformal Willmore functional on constrained small geodesic spheres can be expanded as
$$I(S_{p,\rho}(w))=  \frac{\pi}{5} \Big( \|Ric_p\|^2-\frac{1}{3}  R(p)^2 \Big) \rho^4 + O_p(\rho ^5).$$
A simple computation in the orthonormal basis that diagonalizes $Ric_p$ shows that the first term in the expansion is the squared norm of the Traceless Ricci tensor:
$$\Big( \|Ric_p\|^2-\frac{1}{3}  R(p)^2 \Big)=\|Ric_p-\frac{1}{3}g_p R(p) \|^2=\|S_p\|^2. $$
\end{proof}

\subsection {Proof of the non existence result}\label{subsec:pfNonEx}

We start with a Lemma, which asserts that for small perturbation $u\in C^{4,\a}(S^2)$ and small radius $\rho$, the perturbed geodesic sphere $S_{p,\rh}(u)$ can be obtained as a normal graph on an other geodesic sphere $S_{\tilde{p},\tilde{\rh}}$ with perturbation $\tilde{w}\in {C^{4,\a}}^\perp$: $S_{p,\rh}(u)=S_{\tilde{p},\tilde{\rh}}(\tilde{w})$; for the proof see \cite{Mon} Lemma 5.3.  

\begin{lem}\label{lem:uw}
Let $(M,g)$ be a Riemannian manifold of dimension three and fix $\bar{p}\in M$. Then there exist $B(0,r_1)\subset C^{4,\a}(S^2)$, $\rh_1>0$, a compact neighbourhood $U$ of $\bar{p}$ and three continuous functions 
\\$\cdot$ $p(.): B(0,r_1) \to U\subset M$,
\\$\cdot$ $\rh(.,.):(0,\rh_1) \times B(0,r_1) \to \R^+$, 
\\$\cdot$ $w(.,.):U\times B(0,r_1)  \to C^{4,\a}(S^2)^\perp$, 
\\such that for all $\bar{\rh} < \rh_1$ and $u \in B(0,r_1)$, all the perturbed geodesic  spheres $S_{\bar{p},\bar{\rh}}(u)$  can be realized as
$$S_{\bar{p},\bar{\rh}}(u) = S_{p(u), \rh(\bar{\rho},u)}[w(p(u),u)]. $$
\end{lem}

Now we are in position to prove the non existence result.
\\\\
\begin{pfnb} {\bf of Theorem \ref{TeoNoExist}}.
\\Since $\|S_{\bar{p}}\|\neq 0$, there exists $\eta>0$ and a compact neighbourhood $Z_c$ of $\bar{p}$ such that $\|S_p\|>\eta$ for all $p \in Z_c$.

From Lemma \ref{lemEqAuxContr} there exist $\rho_0>0$ and a ball $B(0,r)\subset C^{4,\a}(S^2)$  such that- for $w \in {C^{4,\a}}^\perp \cap B(0,r) $, $p\in Z_c$ and $\rh<\rh_0$- if the perturbed geodesic sphere $S_{p, \rh}(w)$ is a critical point of $I$ then $w=w_{p,\rh}$ with good decay properties as $\rh\to 0$. Moreover, for $p \in Z_c$ and $\rh<\rh_0$ we can consider the $C^1$ function
$$\Phi(p,\rh)= I(S_{p,\rh}(w_{p,\rh})).$$
Observe that if $S_{\tilde{p},\tilde{\rh}}(w_{\tilde{p},\tilde{\rh}})$ is a critical point for $I$ then a fortiori $(\tilde{p},\tilde{\rh})$ is a critical point of the constricted functional $\Phi(.,.)$.
\\Proposition \ref{svilPhipr} gives an expansion for $\Phi(p,\rh)$; differentiating it with respect to $\rho$ and recalling  (from Lemma \ref{lemEqAuxContr}) that as  $\rho\to 0$ one has $\|w_{p,\rho}\|_{C^{4,\a}}=O(\rho^2)$ and $\| \frac{\partial }{\partial \rho} w_{p,\rho}\|_{L^2}=O(\rho)$ uniformly for $p \in Z_c$, we get
$$\frac{\partial}{\partial \rho} \Phi(p,\rho)=   \frac{4\pi}{5} \|S_p\| \rho^3 + O_p(\rho ^4) $$
and 
\begin{equation}\label{eq:ddrhoPhi}
\Big|\frac{\partial}{\partial \rho} \Phi(p,\rho)  \Big|> \frac{4\pi}{5} \eta \rho^3  +  O(\rho ^4) \quad \text{for all $p \in Z_c$,}
\end{equation}
where the remainder $O(\rho^4)$ is uniform on $Z_c$.
\\From this equation we can say that there exist $\rh_2 \in ]0,\rh_0[$ such that for all $p\in Z_c$ and $\rh<\rh_2$, $(p,\rh)$ is not a critical point of $\Phi$.  
\\Hence 
\begin{eqnarray}
&& \forall w \in C^{4,\a}(S^2)^\perp\cap B(0,r),\text{ } \rh <\rh_2 \text{ and } p\in Z_c \label{hp:wrho} \\ 
&&  \Rightarrow   S_{p,\rh}(w) \text{ is NOT a  critical point of } I. \nonumber
\end{eqnarray}

Now from Lemma \ref{lem:uw}, if $u \in B(0,r_1)\subset C^{4,\a}(S^2)$ and $\bar{\rh}<\rh_1$, any perturbed sphere $S_{\bar{p},\bar{\rh}}(u)$ can be realized as
$$S_{\bar{p},\bar{\rh}}(u) = S_{p(u), \rh(\bar{\rh},u)}[w(p(u),u)],\quad w(p(u),u)\in C^{4,\a}(S^2)^\perp.$$
From the continuity of the functions $p(.)$, $\rho(.,.)$ and $w(.,.)$, there exist $\rho_3 \in ]0,\min(\rh_1,\rh_2)[$ and 
$r_2 \in ]0,\min(r,r_1)[$ such that for all $u\in B(0,r_2)\subset C^{4,\a}(S^2)$ and $\bar{\rh}<\rh_3$ we have:

$\cdot$ $p(u) \in Z_c$,

$\cdot$ $\rh(\bar{\rh},u)<\rh_2$ and

$\cdot$ $w(p(u),u) \in C^{4,\a}(S^2)^\perp\cap B(0,r)$.

It follows that if $u\in B(0,r_2)$ and $\bar{\rh}<\rh_3$, the sphere $S_{\bar{p},\bar{\rh}}(u)$ can be realized as $S_{p(u), \rh(\bar{\rh},u)}[w(p(u),u)]$ which satisfies the assumptions \eqref{hp:wrho}; so it is not a critical point of $I$.
\end{pfnb}

\section{The conformal Willmore functional on perturbed standard spheres $S_p^\rh(w)$ in $(\Rtre, g_\e)$}\label{sec:Ie}
	
Throughout this section  $I_\epsilon(\mathring{M}):=\int_{\mathring{M}} \big[ \frac {H^2}{4}-D \big]d \Sigma_\e$ will be the conformal Willmore functional of the surface $\mathring{M}$ embedded in the ambient manifold $(\Rtre,g_\e)$, where $g_\e=\d+\e h$ is a perturbation of the euclidean metric ($h$ is a bilinear form with good decay properties at infinity, for simplicity we will treat in detail the case when $h$ has compact support but as one can see from the estimates it is enough to take $h$ fast decreasing. See for example \cite{Mon} Theorem 1.1).

The problem will be studied through a perturbation method relying on the Lyapunov-Schmidt reduction: In Subsection \ref{subsec:FDR} we will perform the abstract reduction, in Subsection \ref{subsec:ExpRF} we will compute an expansion of the reduced functional and in the last Subsection \ref{subsec:ProofEx} we will prove the main Theorems of this paper, that is the existence of conformal Willmore surfaces.

\subsection{The finite dimensional reduction}\label{subsec:FDR}
 We already know from Theorem \ref{TeoWillmin} that $I_0$ possesses a critical manifold made up of the standard spheres $S_p^\rho$ of $\Rtre$, we want to study the perturbed functional $I_\e$ near this critical manifold. First of all let us point out a clarification about $I_0'(S_p^\rh)$ and $I_0''(S_p^\rh)$, that are  the first and second variations of the unperturbed functional on the standard spheres, which will be useful throughout this Section.

\begin{rem}\label{rem:I0''}
In the previous paper \cite{Mon}, (remark 3.3, notice the factor difference in the definition of the Willmore functional) we observed that 
$$I_0'(S_p^\rh(w))=\frac {1}{2\rho^3} \triangle_{S^2}(\triangle_{S^2}+2)w+ \frac{1}{\rho^3} Q_p^{(2)(4)}(w)$$
and
$$I_0''(S_p^\rh)[w]=\frac {1}{2\rho^3} \triangle_{S^2}(\triangle_{S^2}+2)[w].$$
The sense of the two formulas were the following.
\\By definition $S_p^\rho(w)$ is a normal graph on $S_p^\rho$ with perturbation $\rh w$ (we chose the inward normal $\mathring{N}$ for all the computations), hence
$$I_0(S_p^\rho(w))=I_0(S_p^\rho)+\int_{S_p^\rho}\big(I_0'(S_p^\rh) (\rh w)\big) d\Sigma_0+ \frac{1}{2} \int_{S_p^\rho} \big(I_0''(S_p^\rh)[w] (\rho w)  \big) d\Sigma_0+o(|w|^2).$$
If we want to bring the expression to the standard sphere we get
$$I_0(S_p^\rho(w))=I_0(S_p^\rho)+\int_{S^2}\big(\rho^3 I_0'(S_p^\rh)  w\big) d\Sigma_0+ \frac{1}{2} \int_{S^2} \big(\rho^3 I_0''(S_p^\rh)[w] w  \big) d\Sigma_0 + o(|w|^2).$$
Now we denote $$\tilde{I}_0'(S_p^\rh(w))=\rho^3 I_0'(S_p^\rh(w))= \frac{1}{2}\triangle_{S^2}(\triangle_{S^2}+2)[w]+ Q_p^{(2)(4)}(w) $$ and  $$\tilde{I}_0'' (S_p^\rh)[w]=\rho^3 I_0''(S_p^\rh)[w]=\frac{1}{2} \triangle_{S^2}(\triangle_{S^2}+2)[w] $$
then we get the more familiar formula
$$I_0(S_p^\rho(w))=I_0(S_p^\rho)+\int_{S^2}\big(\tilde{I}_0'(S_p^\rh)  w\big)+ \frac{1}{2} \int_{S^2} \big(\tilde{I}_0''(S_p^\rh)[w] w  \big)+o(|w|^2).$$
This was about the functional $\int \frac{H^2}{4}$ but the same argument can be repeated for the functional $\int \big(\frac{H^2}{4}-D\big)$ (since the ambient is euclidean, $D=K$ the Gaussian curvature which by the Gauss Bonnet Theorem does not influence the differential).
Since $S_p^\rho$ are critical points for $I_0$ we can say that the conformal Willmore functional on perturbed standard spheres is 
 $$I_0(S_p^\rho(w))=\frac{1}{4} \int_{S^2} \Big( \triangle_{S^2}(\triangle_{S^2}+2)[w] w\Big)+o(|w|^2).	$$
In the following we will always denote
$$I_0'(S_p^\rh)[w]=\frac{1}{2} \triangle_{S^2}(\triangle_{S^2}+2)[w]+ Q_p^{(2)(4)}(w)$$		
$$I_0''(S_p^\rh)[w]=\frac{1}{2} \triangle_{S^2}(\triangle_{S^2}+2)[w]$$						
since, as we saw, it is more natural. 
\end{rem}

Since  from Proposition \ref{svilPhipr} we have an expansion of $I_\e$ on small geodesic spheres and on the other hand the critical manifold of $I_0$ is made up of standard spheres, let us link the two objects. The geodesic sphere in $(\Rtre, g_\e)$ of center $p$ and radius $\rh$ will be denoted by $S_{p,\rh}^\e$.
\begin{lem} \label{lemv}
For small $\e$ the geodesic spheres $S_{p,\rh}^\e$ are normal graphs on the standard spheres $S_p^\rh$ with a perturbation $v_\e \in C^\infty(\R^+ \times \Rtre \times S^2)$: 
$$S_{p,\rh}^\e= S_p^\rh(v_\e(\rh,p,.)).$$
Moreover the perturbation $v_\e$ satisfies the following decreasing properties:
\\1) $\rh v_\e=O(\e)$ in $C^k$ norm on compact subsets of $\R^+ \times \Rtre \times S^2$ for all $k \geq 0$;
\\2) $v_\e(\rh,.,.)=O(\rh)$ as $\rh \to 0$ uniformly for $\Theta \in S^2$ and $p$ in a compact subset of $\Rtre$.
\end{lem}

\begin{proof}
The geodesic spheres $S_{p,\rh}^\e$ are parametrized by $\Theta \mapsto Exp_p(\rho \Theta)$. So one is interested in the solution of the geodesic equation

\begin{displaymath} 
\left\{ \begin{array}{ll}
\ddot{y}^i+\Gamma ^i_{jk} \dot{y}^j \dot{y}^k = 0\\
y^i(0)=p^i \\
\dot{y}^i(0)=\Theta^i \\
\end{array} \right.
\end{displaymath}
evaluated at $\rho$. 
We look for $y^i$ of the form 
$$ y^i= p^i+\rh \Theta^i+ \e u^i +o(\e)$$
where $$u^i: \R^+ \times \Rtre \times S^2 \to \R,\quad (\rho,p,\Theta) \mapsto u^i(\rho,p,\Theta) $$
is $C^\infty (\R^+\times \Rtre \times S^2)$ and have to be determined. A straightforward computation ( setting $\Gamma ^i_{jk}=\e \tilde{\Gamma}^i_{jk}$ ) shows that $u^i$ must solve the following non linear second order ODE:
\begin{displaymath} 
\left\{ \begin{array}{ll}
\ddot{u}^i +\tilde{\Gamma}^i_{jk} \Theta^j \Theta^k=0\\
u^i(0)=0 \\
\dot{u}^i(0)=0 \\
\end{array} \right.
\end{displaymath}
where we have denoted $\dot{u}^i=\frac{\partial}{\partial \rho} u^i$ and $\ddot{u}^i=\frac{\partial^2}{\partial \rho^2} u^i$ and the equation has to be considered at $(p,\Theta)$ fixed.
\\Since $h$ is compactly supported (more generally it is enough to assume that $h$  and its first derivatives vanish at infinity), the Christoffel symbols $\tilde{\Gamma}^i_{jk}$ vanish at infinity and the ODE admits unique solution defined for all $\rho \geq 0$. From differentiable dependence on parameters, $u^i$ is of class $C^\infty (\R^+\times \Rtre \times S^2)$, observe also that $u^i=O(\rho^2)$ as $\rho \to 0$ uniformly for $\Theta \in S^2$ and $p$ in a compact subset of $\Rtre$.
\\It follows that the geodesic sphere $S_{p,\rh}^\e$ can be obtained from the standard sphere $S_p^\rh$ with the small variation $\e u^i(\rh,p, \Theta)$. Now it is easy to see that for $\e$ small enough there exists $v_\e \in C^\infty(\R^+ \times \Rtre \times S^2)$ such that 
\begin{itemize}
	\item $S_{p,\rho}^\e=S_p^\rho (v_\e)$ 
	\item $\rh v_\e =O(\e)$ in $C^k$ norm on compact subsets  of $\R^+\times \Rtre \times S^2$ for all $k \geq 0$ 
	\item $v_\e(\rh,\Theta)=O(\rh)$ as $\rh \to 0$, uniformly for $\Theta \in S^2$ and $p$ in a compact subset of $\Rtre$.
\end{itemize}

\end{proof}

Now we define the manifold of approximate solutions that will play the role of the ``critical manifold'' $Z$.
Let $R_1$ and $R_2$ be positive real numbers to be determined and $\chi$ a $C^\infty(\R^+)$ cut off function such that

\begin{displaymath} 
\left\{ \begin{array}{ll}
\chi(\rh)=1 \text{  for $0\leq\rh\leq R_1$} \\
0\leq \chi(\rh)\leq 1 \text{  for $R_1\leq\rh\leq R_2$} \\
\chi(\rho)=0  \text{  for $\rh \geq R_2$}. 
\end{array} \right.
\end{displaymath}
We denote with $\Sigma_{p,\rh}^\e$ the perturbed standard sphere
\begin{equation}\label{def:Sigma}
\Sigma_{p,\rh}^\e= S_p^\rh(\chi v_\e)
\end{equation}
and we consider it as parametrized on $S^2$; observe that for $\rho<R_1$ one gets the geodesic spheres $\Sigma_{p,\rh}^\e=S_{p,\rho}^\e$ and for $\rho>R_2$ one has the standard spheres $\Sigma_{p,\rh}^\e=S_p^\rho$.
\\Denoted by $\mathring{N}$ the inward normal unit vector, given a  function $w$ on $S^2$,  
$\Sigma_{p,\rh}^\e(w)$ will be the surface parametrized by $\Sigma_{p,\rh}^\e+\rh w \mathring{N}$ (notice that we are consistent with the previous notations since $\Theta$ points outward).

At this point we can state the two Lemmas which allow us to perform the Finite Dimensional Reduction. Recall that, as always, $P:L^2(S^2)\to Ker[\triangle_{S^2}(\triangle_{S^2}+2)]^\perp$ is the orthogonal projection.

\begin{lem} \label{lemw}
For each compact subset $Z_c\subseteq  \Rtre\oplus \R^+$, there exist $\epsilon_0>0$ and $r>0$ with the following property: for all $|\epsilon|\leq \epsilon_0$ 
and $(p,\rho)\in Z_c$, the auxiliary equation $PI_\epsilon'(\Sigma_{p,\rh}^\e(w))=0$ has unique solution $w=w_\epsilon (p,\rho)\in B(0,r) \subset C^{4,\a}(S^2)^\perp$ such that:
\\1) the map $w_\epsilon (.,.): Z_c \to C^{4,\alpha}(S^2)^\perp $ is of class $C^1$; 
\\2) $\|w_\epsilon (p,\rho)\|_{C^{4,\alpha}(S^2)}\to 0$ for 
$\epsilon \to 0$ uniformly with respect to $(p,\rho) \in Z_c$;
\\3) more precisely $\|w_\epsilon (p,\rho)\|_{C^{4,\alpha}(S^2)}= O(\epsilon)$ for 
$\epsilon \to 0$ uniformly in $(p,\rho) \in Z_c$;
\\4) $\|w_\e(p,\rh)\|_{C^{4,\a}}=O(\rh^2)$ uniformly for $p$ in the compact set. 
\end{lem}

\begin{proof}
The proof will be rather sketchy, for more details we refer to Section 4 of \cite{Mon}.

$\bullet$ $\rh\leq R_1$:  Recall Lemma \ref{lemEqAuxContr} and choose $R_1=\rho_0$;  for $\rh \leq R_1$, the surface $\Sigma_{p,\rh}^\e$ coincides with the geodesic sphere $S_{p,\rh}$, so thanks to Lemma \ref{lemEqAuxContr} there exists a unique  $w_\e(p,\rh) \in C^{4,\a}(S^2)^\perp$ which solves the auxiliary equation.
During the proof of Proposition \ref{lem:svilI'} we wrote $I'$ as in equation \eqref{eq:diffI}; observing that all the curvature tensors of $(\Rtre, g_\e)$ are of order $O(\e)$ (in $C^k$ norm $\forall k \in \N$ on each fixed compact set of $\Rtre$),  it follows that $$PI_\e'(S_{p,\rh}^\e(w_\e(p,\rh)))=\frac{1}{2}\triangle_{{S_{p,\rh}^\e(w_\e(p,\rh))}}H+Q^{(2)(4)}(w_\e(p,\rh))+O(\e)=0 \quad \text{in }C^{0,\a}(S^2);$$  
from this formula and the expansions of $\mathring{h}$, $\mathring{g}^{-1}$ and $H$, we have that 
$$\triangle_{S^2}(\triangle_{S^2}+2)[w_\e(p,\rh)]+Q^{(2)(4)}(w_\e(p,\rh))=O(\e) \quad \text{in }C^{0,\a}(S^2)$$ 
uniformly for $(p,\rh) \in Z_c$; first observe that $\|w_\e\|_{C^{4,\a}(S^2)}\to 0$ as $\e \to 0$ uniformly in $Z_c$ so the second summand is negligible, then conclude that  $\|w_\e\|_{C^{4,\a}(S^2)}=O(\e)$ uniformly on $Z_c$. The other properties follow from Lemma \ref{lemEqAuxContr}.

$\bullet$ $\rh\geq R_2$: in this case the surface $\Sigma_{p,\rh}^\e$ coincides with the standard sphere $S_p^\rh$ for which the discussion has already been done in Lemma 4.1 of \cite{Mon}.

$\bullet$ $R_1\leq\rh \leq R_2$: with a Taylor expansion the auxiliary equation becomes
$$0=PI_\e'(\Sigma_{p,\rh}^\e(w_\e))=PI_\e'(\Sigma_{p,\rh}^\e)+PI_\e''(\Sigma_{p,\rh}^\e)[w_\e]+o(\|w_\e\|_{C^{4,\a}(S^2)}).$$
But by definition $\Sigma_{p,\rh}^\e=S_p^\rh(\chi v_\e)$, so
$$I_\e'(\Sigma_{p,\rh}^\e)= I_\e'(S_p^\rh(\chi v_\e))=I_0'(S_p^\rh)+I_0''(S_p^\rh)[\chi v_\e]+O(\e). $$
Since $I_0'(S_p^\rh)=0$ and $\|v_\e\|_{C^{4,\a}}=O(\e)$ we get
$$\|I_\e'(\Sigma_{p,\rh}^\e)\|_{C^{0,\a}(S^2)}=O(\e).$$
Now $PI_0''(S_p^ \rh)=\triangle_{S^2}(\triangle_{S^2}+2)$ which is an invertible map ${C^{4,\a}}^\perp \to {C^{0,\a}}^\perp$ uniformly on $Z_c$; since the set of invertible operators is open, for $\e$ small also $PI_\e''(S_p^ \rh)$ is uniformly invertible. From the fact that $\|v_\e\|_{C^k(S^2)}=O(\e)$ for all $k$ it follows that also $PI_\e''(\Sigma_{p,\rh}^\e)=PI_\e''(S_p^ \rh(\chi v_\e ))$ is uniformly invertible on $Z_c$.
With a fixed point argument analogous to the proof of Lemma 4.1 in \cite{Mon} it is possible to show that there exist $r>0$ and a unique solution $w_\e\in B(0,r)\subset {C^{4,\a}}^\perp$  of
$$w_\e=-PI_\e''(\Sigma_{p,\rh}^\e)^{-1} \Big( PI_\e'(\Sigma_{p,\rh}^\e) +o(\|w_\e\|_{C^{4,\a}(S^2)})\Big)$$
with the desired properties.
\end{proof}
 
Now we are in position to define the reduced functional $\Phi_\e(p,\rh)=I_\e(\Sigma_{p,\rh}^\e(w_\e(p,\rh)))$ and to state the following fundamental Lemma:

\begin{lem} \label{lemRFD}
Fixed a compact set $Z_c\subseteq  \Rtre\oplus \R^+$, for $|\epsilon|\leq \epsilon_0$ consider the functional $\Phi_\epsilon: Z_c\to \R$.
\\Assume that, for $\epsilon$ small enough, $\Phi_\epsilon$ has a critical point $(p_\epsilon ,\rho_\epsilon ) \in Z_c$. Then 
$\Sigma_{p_\epsilon,\rho_\epsilon}^\e (w_\epsilon (p_\epsilon, \rho_\epsilon ))$ is a critical point of $I_\epsilon$.
\end{lem} 

\begin{proof}
The proof is a slight modification of the proof of Lemma 4.2 in \cite{Mon} just using the good decay properties of $v_\e$, $w_\e$ and their derivatives as $\e \to 0$.
\end{proof}

\begin{rem}\label{rem:PhiR}
The reduced functional $\Phi_\e$ is defined for small $\e$ once a compact $Z_c \subset \Rtre \oplus \R^+$ is fixed. In the following discussion we will study the behaviour of $\Phi_\e$ for large $\rho$; this makes sense since the compact $Z_c$ can be chosen arbitrarily large and the solution of the auxiliary equation $w_\e(p,\rh)$ given in Lemma \ref{lemw} is unique in a small ball of $C^{4,\a}(S^2)^{\perp}$. However the compact $Z_c$ will be chosen in a rigorous and appropriate way in the proofs of Theorem \ref{teoExist} and Theorem \ref{teoExist2}.  
\end{rem}

\subsection{Expansion of the reduced functional $I_\e(\Sigma_{p,\rh}^\e (w_\e(p,\rh)))$}\label{subsec:ExpRF}

Since Lemma \ref{lemw}  applies, we can perform the Finite Dimensional Reduction. In this Subsection we will study the reduced functional 
$\Phi_\e(p,\rho)=I_\e(\Sigma_{p,\rho}^\e(w_\e(p,\rh)))$.
For $\rho<R_1$, $\Sigma_{p,\rho}^\e=S_{p,\rho}^\e$ so for small radius $\rho$ we have the explicit expansion of $\Phi_\e(p,\rho)=I_\e(S_{p,\rho}^\e(w_\e(p,\rho)))$  given by Proposition \ref{svilPhipr}.
More generally, for all the radius we can write the conformal Willmore functional on our surfaces $\Sigma_{p,\rh}^\e (w)$ as

\begin{equation}\label{eq:IG}
I_\e(\Sigma_{p,\rh}^\e (w))=I_0(\Sigma_{p,\rh}^\e (w))+\e G_1(\Sigma_{p,\rh}^\e (w))+\e^2 G_2(\Sigma_{p,\rh}^\e (w))+o(\e^2).
\end{equation}

Now let us study the case $\rho>R_2$, when $\Sigma_{p,\rh}^\e =S_p^\rho$; in this circumstance we get the formula
\begin{equation}\label{eq:IGS}
I_\e(S_p^\rho (w))=I_0(S_p^\rho (w))+\e G_1(S_p^\rho (w))+\e^2 G_2(S_p^\rho (w))+o(\e^2). 
\end{equation}

\begin{lem}\label{lem:G1}
For all standard spheres $S_p^\rho$ one has 
$$I_0(S_p^\rho)=G_1(S_p^\rho)=0.$$
\end{lem}

\begin{proof}
As above, we write the functional as $I_\e(S_p^\rho)=I_0(S_p^\rho)+\e G_1(S_p^\rho)+o(\e)$.
First let us expand in $\epsilon$ the geometric quantities of interest starting from the area form $d\Sigma _\epsilon:=\sqrt{E_\epsilon G_\epsilon-F_\epsilon^2}$.
\begin{eqnarray}
E_\epsilon &=& g_\e(\theta _1,\theta _1)=(\theta _1,\theta _1)+\epsilon  h(\theta _1, \theta _1)=E_0+\epsilon  h(\theta _1, \theta _1) \nonumber \\
F_\epsilon &=&F_0+\epsilon  h(\theta _1, \theta _2)=\epsilon  h(\theta _1, \theta _2) \nonumber \\
G_\epsilon &=&G_0+\epsilon  h(\theta _2, \theta _2), \nonumber
\end{eqnarray}
where $(.,.)$ denotes the euclidean scalar product and $E_0,F_0,G_0$ are the coefficients of the first fundamental form in euclidean metric. 
The area form can be expanded as
\begin{eqnarray}
d\Sigma _\epsilon &:=&\sqrt{E_\epsilon G_\epsilon-F_\epsilon^2}\nonumber \\
&=&\sqrt{E_0 G_0+ \epsilon \big( E_0 h(\theta _2, \theta _2)+G_0 h(\theta _1, \theta _1)\big) +o(\epsilon) },\nonumber 
\end{eqnarray} 
where the remainder $o(\epsilon)$ is uniform fixed the compact set in the variables $(p,\rho), \rho>0$.
\\Using the standard Taylor expansion $\sqrt{a+bx+cx^2}=\sqrt{a}+\frac{1}{2} \frac{b}{\sqrt{a}}x+o(x)$, we get
\begin{equation}\label{svilFormaArea}
\sqrt{E_\epsilon G_\epsilon- F_\epsilon ^2}=\sqrt{E_0 G_0}+\frac{\epsilon}{2} \frac{E_0 h(\theta _2, \theta _2)+G_0 h(\theta _1, \theta _1)}{\sqrt{E_0 G_0}}+o(\epsilon ),
\end{equation}
where the remainder $o(\epsilon)$ is uniform fixed the compact set in the variables $(p,\rho)$.

Now let us expand the second fundamental form. 
\\First of all we have to find an expression of the inward normal unit vector  $\nu_\epsilon$ on $S_p ^\rho$ in metric $g_\epsilon$.
\\We look for $\nu_\epsilon$ of the form
$$\nu_\epsilon=\nu_0+\epsilon N+ o(\epsilon) $$
where $\nu_0=-\Theta$ is the inward normal unit vector on $S_p ^\rho$ in euclidean metric and the remainder is $o(\epsilon)$ uniformly fixed the compact in $(p,\rho)$.
From the orthogonality conditions $g_\epsilon(\theta_1, \nu_\epsilon)=0$ and $g_\epsilon(\theta_2, \nu_\epsilon)=0$, we get
\begin{eqnarray}
0&=&g_\epsilon(\theta _1,\nu _\epsilon )=(\theta _1,\nu _0)+ \epsilon (\theta_1, N)+\epsilon h(\theta _1,\nu _0)+o(\epsilon) \nonumber \\
0&=&g_\epsilon(\theta _2,\nu _\epsilon )=(\theta _2,\nu _0)+ \epsilon (\theta_2, N)+\epsilon h(\theta _2,\nu _0)+o(\epsilon) \nonumber 
\end{eqnarray}
from which, being $\nu_0$ the euclidean normal vector to $S_p ^\rho$,
\begin{eqnarray}
(N,\theta _1)&=&-h(\nu _0 ,\theta _1) \label{N,theta1} \\
(N,\theta _2)&=&-h(\nu _0 ,\theta _2) \label{N,theta2}.
\end{eqnarray}
Imposing the normalization condition on $\nu_\epsilon$ we obtain
$$1=g_\epsilon (\nu _\epsilon ,\nu _\epsilon )=(\nu _0,\nu _0)+2 \epsilon (\nu _0,N)+\epsilon h(\nu _0,\nu _0)+ o(\epsilon )$$
from which, being $(\nu_0,\nu_0)=1$
\begin{equation}\label{N,nu0}
(N,\nu _0)= - \frac{1}{2}h(\nu _0 ,\nu _0).
\end{equation}
Denote with $\bar{\theta}_i=\frac{\theta_i}{|\theta_i|}$ the normalized tangent vectors; since $(\bar{\theta}_1,\bar{\theta}_2,\nu_0)$ are an orthonormal base, the expressions  \eqref{N,theta1},\eqref{N,theta2},\eqref{N,nu0} characterize univocally $N$, which can be written in this base as
\begin{equation}\label{eqN}
N= -h(\nu _0 ,\bar{\theta}_1)\bar{\theta}_1-h(\nu _0 ,\bar{\theta}_2)\bar{\theta}_2- \frac{1}{2}h(\nu _0 ,\nu _0)\nu _0.
\end{equation}
Knowing the normal vector we can evaluate the coefficients of the second fundamental form 
$$\mathring{h}_{\epsilon ij} :=-g_\epsilon (\nabla_{\theta_i} \nu_\epsilon, \theta_j),$$
where $\nabla$ is the connection on $\Rtre$ endowed with the metric $g_\epsilon$.
By linearity, denoting with $\frac{\partial }{\partial x^\lambda}$ the standard euclidean frame  of $\Rtre$
\begin{eqnarray}
\nabla_{\theta_i} \nu_\epsilon &=& \theta _i ^\mu \nabla _\mu (\nu _\epsilon ^\lambda \frac{\partial }{\partial x^\lambda} ) =  \frac{\partial \nu _\epsilon }{\partial \theta ^i} + \theta _i ^\mu \nu _\epsilon ^\lambda \Gamma _{\mu \lambda} ^\nu \frac{\partial }{\partial x^\nu} \nonumber \nonumber
\end{eqnarray} 
where $\Gamma _{\mu \lambda} ^\nu$ are the Christoffel symbols of $(\Rtre,g_\epsilon)$.
\\Let us find an expansion in $\epsilon$ of  $\Gamma _{\mu \lambda} ^\nu$. By definition
$$\Gamma _{\mu \lambda} ^\nu = \frac{1}{2}g^{\nu\sigma } [D_\mu g_{\lambda \sigma}+D_\lambda g_{\sigma\mu}-D_\sigma g _{\mu \lambda}].$$
Noticing that $g^{\mu\sigma }=\delta^{\mu\sigma }-\epsilon h_{\mu\sigma }+o(\epsilon)$ and $D_\mu g_{\lambda \sigma}=\epsilon D_\mu h _{\lambda \sigma}$, we obtain
\begin{eqnarray}
\Gamma _{\mu \lambda} ^\nu &=& \frac{1}{2}\epsilon \delta^{\nu\sigma } [D_\mu h_{\lambda \sigma}+D_\lambda h_{\sigma\mu}-D_\sigma h_{\mu \lambda}]+o(\epsilon) \nonumber\\
&=&\frac{1}{2}\epsilon \delta^{\nu\sigma } A_{\mu\sigma\lambda} \nonumber\\
\end{eqnarray}
where we set
\begin{equation}\label{defA}
A_{\mu \nu\lambda}:= [D_\mu h_{\lambda \nu}+D_\lambda h_{\nu\mu}-D_\nu h_{\mu \lambda}].
\end{equation}
Hence
$$\nabla_{\theta_i} \nu_\epsilon = \frac{\partial \nu _\epsilon }{\partial \theta ^i} + \frac{1}{2}\epsilon \theta _i ^\mu \nu _0 ^\lambda  \delta^{\nu\sigma } A_{\mu \sigma \lambda}\frac{\partial }{\partial x^\nu}+o(\epsilon)$$
and the second fundamental form becomes
\begin{equation}\label{formulah}
\mathring{h}_{\epsilon ij}=-\bigg(\frac{\partial \nu _0 }{\partial \theta ^i},\theta _j \bigg)-\epsilon \bigg[ h\bigg(\frac{\partial \nu _0 }{\partial \theta ^i}, \theta _j\bigg) +\bigg(\frac{\partial N }{\partial \theta ^i}, \theta _j\bigg) \bigg]-\frac{1}{2}\epsilon\theta _i ^\mu \theta _j ^\nu \nu _0 ^\lambda  A_{\mu\nu\lambda}.
\end{equation}

In order to simplify the expressions let us recall the values of the coefficients of the unperturbed first fundamental form 
\begin{eqnarray}
E_0&=&\rho^2 \nonumber \\
F_0&=&0  \nonumber \\
G_0&=&\rho^2 \sin^2 \theta ^1, \nonumber 
\end{eqnarray}
those of the unperturbed second fundamental form (following the classical notation of the theory of surfaces, we denote with $l_0,m_0,n_0$ the quantities $\mathring{h}_{0_{11}},\mathring{h}_{0_{12}},\mathring{h}_{0_{22}}$) 
\begin{eqnarray}
l_0&=&\rho \nonumber\\
m_0&=&0 \nonumber\\
n_0&=& \rho \sin ^2\theta ^1\nonumber
\end{eqnarray}
and the unperturbed mean curvature and Gaussian curvature 
$$H_0= \frac{2}{\rho}$$
$$D_0= \frac{1}{\rho^2}.$$

From formula (22) in the proof of Lemma 3.4 in \cite{Mon} and the above expressions of the unperturbed quantities we have immediately that

\begin{eqnarray}
\int_{S_p^\rho} \frac{H_\e^2}{4}d\Sigma_\e&=&4\pi-\frac{1}{2}\epsilon \int _{S^2} \Big[ h(\bar{\theta _1}, \bar{\theta _1})+h(\bar{\theta _2}, \bar{\theta _2}) \Big] d\Sigma_0 -\frac{1}{2} \epsilon \rho  \int _{S^2} \Big[ (\bar{\theta} _2 ^\mu \bar{\theta} _2 ^\nu +  \bar{\theta} _1 ^\mu \bar{\theta} _1 ^\nu ) \nu _0 ^\lambda  A_{\mu\nu\lambda} \Big] d\Sigma _0\nonumber\\
&&-\epsilon \int _{(0,\pi)\times (0,2\pi)}  \bigg[ h\bigg(\frac{\partial \nu _0 }{\partial \theta ^2},\bar{\theta _2}\bigg)+\bigg(\frac{\partial N }{\partial \theta^2}, \bar{\theta _2}\bigg)\bigg] d\theta^1 d\theta ^2  -\epsilon \int _{S^2} \bigg[ h\bigg(\frac{\partial \nu _0 }{\partial \theta ^1},\bar{\theta _1}\bigg) +\bigg(\frac{\partial N }{\partial \theta^1}, \bar{\theta _1}\bigg) \bigg] d\Sigma _0+o(\e). \label{intH24}
\end{eqnarray} 
Now we have to compute $\int_{S_p^\rho} D_{\e} d\Sigma_{\e}$. Knowing the first and the second fundamental forms we can evaluate $D_{\e}:=\frac{\det{\mathring{h}_\e}}{\det{\mathring{g}_\e}}$, in fact observing that 
\begin{eqnarray}
\det{\mathring{h}_\epsilon}&=&\det{\mathring{h}_0}-\e n_0 \bigg[ h\bigg(\frac{\partial \nu _0 }{\partial \theta ^1}, \theta _1\bigg)
+\bigg(\frac{\partial N }{\partial \theta^1}, \theta _1\bigg) \bigg] -\frac{1}{2}\epsilon n_0 \theta _1 ^\mu \theta _1 ^\nu \nu _0 ^\lambda  A_{\mu\nu\lambda}+ \nonumber \\
&&-\e l_0 \bigg[ h\bigg(\frac{\partial \nu _0 }{\partial \theta ^2}, \theta _2\bigg)
+\bigg(\frac{\partial N }{\partial \theta^2}, \theta _2\bigg) \bigg] -\frac{1}{2}\epsilon l_0 \theta _2 ^\mu \theta _2 ^\nu \nu _0 ^\lambda  A_{\mu\nu\lambda}+o(\e)
\end{eqnarray}
and that
$$\det{\mathring{g}_\e}=\det{\mathring{g}_0}+\e E_0 h(\theta_2,\theta_2)+\e G_0h(\theta_1,\theta_1),$$
using the Taylor expansion $\frac{1}{a+\epsilon b+o(\epsilon)}=\frac{1}{a}-\epsilon \frac{b}{a^2}+o(\epsilon)$, we get
\begin{eqnarray}
D_\epsilon &=& D_0-\e \frac{ n_0 \bigg[ h\bigg(\frac{\partial \nu _0 }{\partial \theta ^1}, \theta _1\bigg)
+\bigg(\frac{\partial N }{\partial \theta^1}, \theta _1\bigg) \bigg] +\frac{1}{2} n_0 \theta _1 ^\mu \theta _1 ^\nu \nu _0 ^\lambda  A_{\mu\nu\lambda}+l_0 \bigg[ h\bigg(\frac{\partial \nu _0 }{\partial \theta ^2}, \theta _2\bigg)
+\bigg(\frac{\partial N }{\partial \theta^2}, \theta _2\bigg) \bigg] +\frac{1}{2} l_0 \theta _2 ^\mu \theta _2 ^\nu \nu _0 ^\lambda  A_{\mu\nu\lambda}}{E_0 G_0} \nonumber \\
&&-\e \frac{\big[E_0 h(\theta _2, \theta _2)+G_0 h(\theta _1, \theta _1)\big]\det{\mathring{h}_0}}{(E_0 G_0)^2}+o(\e). \nonumber 
\end{eqnarray}
Recalling \eqref{svilFormaArea} we obtain

\begin{eqnarray}
\int_{S_p^\rho} D_\e d\Sigma _\e &=&\int_{S_p^\rho} D_0 d\Sigma_0 + \frac{\epsilon}{2} \int_{(0,\pi)\times(0,2\pi)} D_0 \bigg\{ \frac{E_0 h(\theta _2, \theta _2)+G_0 h(\theta _1, \theta _1)}{\sqrt{E_0 G_0}} \bigg\} d\theta^1 d\theta ^2 \nonumber\\
&&- \e  \int_{(0,\pi)\times(0,2\pi)} \Bigg\{ \frac{ n_0 \bigg[ h\bigg(\frac{\partial \nu _0 }{\partial \theta ^1}, \theta _1\bigg)
+\bigg(\frac{\partial N }{\partial \theta^1}, \theta _1\bigg) \bigg] +\frac{1}{2} n_0 \theta _1 ^\mu \theta _1 ^\nu \nu _0 ^\lambda  A_{\mu\nu\lambda}+l_0 \bigg[ h\bigg(\frac{\partial \nu _0 }{\partial \theta ^2}, \theta _2\bigg)
+\bigg(\frac{\partial N }{\partial \theta^2}, \theta _2\bigg) \bigg] }{\sqrt{E_0 G_0}} \Bigg\} d\theta^1 d\theta^2 \nonumber \\
&&-\e \int_{(0,\pi)\times(0,2\pi)} \Bigg\{ \frac{\frac{1}{2} l_0 \theta _2 ^\mu \theta _2 ^\nu \nu _0 ^\lambda  A_{\mu\nu\lambda}}{\sqrt{E_0 G_0}} + \frac{\big[E_0 h(\theta _2, \theta _2)+G_0 h(\theta _1, \theta _1)\big]\det{\mathring{h}_0}}{(E_0 G_0)^{3/2} }\Bigg\} d\theta^1 d\theta^2+o(\epsilon).\nonumber\\ \label{intDe}
\end{eqnarray}

Plugging the unperturbed quantities into \eqref{intDe}, after some easy computations we get
\begin{eqnarray}
\int_{S_p^\rho}D_\e d\Sigma_\e &=& 4\pi-\frac{1}{2} \e\int_{S^2} \big[ h(\bar{\theta}_2, \bar{\theta}_2)+ h(\bar{\theta}_1, \bar{\theta}_1)\big] d\Sigma_0 - \e \int_{(0,\pi)\times(0,2\pi)} \bigg[ h\bigg(\frac{\partial \nu _0 }{\partial \theta ^2}, \bar{\theta}_2\bigg)
+ \bigg(\frac{\partial N }{\partial \theta^2}, \bar{\theta}_2\bigg) \bigg] d\theta^1 d\theta^2 \nonumber \\
&&-\e\int_{S^2} \Bigg\{ \bigg[ h\bigg(\frac{\partial \nu _0 }{\partial \theta ^1}, \bar{\theta}_1\bigg)
+\bigg(\frac{\partial N }{\partial \theta^1},\bar{\theta}_1 \bigg) \bigg] +\frac{\rho}{2}[ \bar{\theta}_1^\mu \bar{\theta}_1^\nu \nu _0 ^\lambda  A_{\mu\nu\lambda}+ \bar{\theta}_2 ^\mu \bar{\theta}_2 ^\nu \nu _0 ^\lambda  A_{\mu\nu\lambda}] \Bigg\}d\Sigma_0 +o(\e).\label{int:De}
\end{eqnarray}
Comparing the integrals \eqref{intH24} and \eqref{int:De} we see that all terms cancel out and we can conclude that
$$\int_{S_p^\rho}\Big[ \frac{H_\e^2}{4}-D_\e\Big] d\Sigma_\e=o(\e).$$
\end{proof}

In the following Lemma we find the expansion of the reduced functional $\Phi_\e$ in terms of $I_0, G_1, G_2$ and their derivatives. Recall the notation introduced in Remark \ref{rem:I0''} about $I_0'$ and $I_0''$ and the definition of $R_2$ given in the Subsection \ref{subsec:FDR} after Lemma \ref{lemv}.

\begin{lem}\label{lemsvile2}
For $\rh>R_2$ the reduced functional has the following expression:
$$ \Phi_\e= \e^2 \bigg(G_2(S_p^\rh)-\frac{1}{2}\int_{S^2}\bigg[G_1'(S_p^\rh) \big(I_0''(S_p^\rh)^{-1} \big[G_1'(S_p^\rh)  \big]\big)  \bigg]d\Sigma_0 \bigg)+o(\e^2). $$
\end{lem}

\begin{proof}
With a Taylor expansion in $\e$, $w$ and recalling that $\|w\|_{C^{4,\a}}=O(\e)$ (see Lemma \ref{lemw}), we have
\begin{eqnarray}
I_\e'(S_p^\rh(w))&=&I_0'(S_p^\rh(w))+\e G_1'(S_p^\rh(w))+o(\e)\nonumber\\
&=&I_0'(S_p^\rh)+I_0''(S_p^\rh)[w]+\e G_1'(S_p^\rh)+o(\e). \nonumber
\end{eqnarray}
Since $I_0'(S_p^\rh)=0$ and $w$ satisfies the auxiliary equation $PI_\e'(S_p^\rh(w))=0$, we must have 
$$w=-\e I_0''(S_p^\rh)^{-1} [PG_1'(S_p^\rh)].$$
Observe that from $G_1(S_p^\rh)\equiv 0$ $\forall p,\rho$ it follows that $G_1'(S_p^\rh)\in Ker[\triangle_{S^2}(\triangle_{S^2}+2)]^\perp$, so $PG_1'(S_p^\rh)=G_1'(S_p^\rh)$.
\\Hence, recalling that $I_0(S_p^\rh)=0, I_0'(S_p^\rh)=0, G_1(S_p^\rh)=0$ we have 
\begin{eqnarray}
I_\e(S_p^\rh(w))&=& I_0(S_p^\rh(w))+\e G_1(S_p^\rh(w))+\e^2 G_2(S_p^\rh(w))+o(\e^2)\nonumber\\
&=& \frac{1}{2}\int_{S^2}\Big[I_0''(S_p^\rh)[w] \quad w\Big]d\Sigma_0 +\e \int_{S^2} \big[ G_1'(S_p^\rh) \quad w \big]d\Sigma_0+ \e^2 G_2(S_p^\rh)+o(\e^2) \nonumber \\
&=& -\frac{1}{2} \e^2 \int_{S^2}\Big[ G_1'(S_p^\rh) \quad I_0''(S_p^\rh)^{-1} [G_1'(S_p^\rh)]\Big] d\Sigma_0+ \e^2 G_2(S_p^\rh)+o(\e^2). \nonumber
\end{eqnarray}

\end{proof}

Now we want to estimate the quantities  $G_1'(S_p^\rh)$ and $G_2(S_p^\rh)$ appearing in the expression of the reduced functional.
\begin{lem}\label{lem:G1'G2}
Writing the conformal Willmore functional on perturbed standard spheres as in \eqref{eq:IGS}, we get the following expressions for the differential of $G_1$ and for $G_2$ evaluated on $S_p^\rh$:
$$G_1'(S_p^\rho)= L(h)+ (1+\rh) \big[ L(Dh)+L(D^2h)+L(D^3h) \big]$$
$$G_2(S_p^\rho)=\int_{S_p^\rho} \Big[  \frac{1}{\rho^2} L(h)L(Dh)+ \frac{1}{\rh} L(h) L(Dh) +\frac{1}{\rho^2}(Q(h)+Q(Dh))+ \frac{1}{\rho} Q(Dh) + Q(Dh)  \Big]$$
where $L(.)$ and $Q(.)$ denote a generic linear (respectively quadratic)  function in the entries of the matrix argument with smooth coefficients on $S^2$ which can change from formula to formula and also in the same formula.
\end{lem}

\begin{proof}
To get the expression of the desired quantities we compute the expansion of $I_\e (S_p^\rh)$ at second order in $\e$ and first order in $w$. In the intention of simplifying the notation, we will omit the remainder terms in the expansions. During the proof we use $L(.)$ and $Q(.)$ to denote a generic linear (respectively quadratic) in the components real, vector or matrix-valued function, with real, vector or matrix argument and with smooth coefficients on $S^2$. The letter $a$ will denote a smooth real, vector or matrix-valued function on $S^2$. $L,Q$ and $a$ can change from formula to formula and also in the same formula. 

Let us start with the expansion. Observe that $S_p^\rh$ is parametrized by $p+\rh (1-w) \Theta$ so the tangent vectors are 
$$Z_i= \rh (1-w) \Theta_i- \rh w_i \Theta=\rh (a+L(w)+L(Dw)).$$
The first fundamental form on $S_p^\rh$ is 
$$\mathring{g}_{ij}=g_\e(Z_i,Z_j)=(Z_i,Z_j)+\e h(Z_i,Z_j)= \rh^2 \Big[ a+ L(w)+L(Dw) +\e L(h)(a+ L(w)+L(Dw)) \Big]$$
and
$$\det{\mathring{g}}=\rh^4 \Big[ a+ L(w)+L(Dw) +\e L(h)(a+ L(w)+L(Dw)) +\e^2 Q(h)\Big]$$
$$\sqrt{\det{\mathring{g}}}=\rh^2 \Big[ a+ L(w)+L(Dw) +\e L(h)(a+ L(w)+L(Dw)) +\e^2 Q(h)\Big];$$
it's easy to see that the inverse of metric is
$$\mathring{g}^{ij}=\frac{1}{\rh^2} \Big[ a+ L(w)+L(Dw) +\e L(h)(a+ L(w)+L(Dw)) +\e^2 Q(h)\Big].$$
The normal versor $\nu_\e$ has to satisfy the three following equations:
$$0=g_\e(\nu_\e, Z_i)=(\nu_\e,Z_i)+\e h(\nu_\e,Z_i)=\nu_\e (1+\e L(h)) \Big( a+L(w)+L(Dw)\Big)$$
$$1=g_\e(\nu_\e , \nu_\e).$$
Hence, just solving the linear system given by the first two conditions and plugging in the third one, we realize that
$$\nu_\e=a+L(w)+L(Dw)+ \e L(h) (a+L(w)+L(Dw)) +\e^2 Q(h).$$
In order to compute the second fundamental form $\mathring{h}_\e=-g_\e(\nabla_{Z_i}\nu_\e, Z_j)$ recall that
$$\nabla_{Z_i}\nu_\e=\frac{\partial \nu_\e}{\partial \theta^i}+ Z_i^\mu \nu_\e^\lambda \Gamma_{\mu \lambda}^{\nu} \frac{\partial}{\partial x^\nu}$$
and that
$$\Gamma _{\mu \lambda} ^\nu = \frac{1}{2}\epsilon \delta^{\nu\sigma } [D_\mu h_{\lambda \sigma}+D_\lambda h_{\sigma\mu}-D_\sigma h_{\mu \lambda}]=\e L(Dh),$$
so the covariant derivative of $\nu_\e$ can be written as
\begin{eqnarray}
\nabla_{Z_i}\nu_\e&=&a+L(w)+L(Dw)+L(D^2 w)+ \e L(Dh) (a+L(w)+L(Dw))+\e L(h)(a+L(w)+L(Dw)+L(D^2w)) \nonumber\\
&&+ \e \rh L(Dh)\big(a+L(w)+L(Dw)\big)+\e^2(1+\rh) L(h)L(Dh)  \nonumber 
\end{eqnarray}
and the second fundamental form becomes
\begin{eqnarray}
\mathring{h}_\e&=&\rh \Big[a+L(w)+L(Dw)+L(D^2 w)+ \e L(Dh) (a+L(w)+L(Dw))+\e L(h)(a+L(w)+L(Dw)+L(D^2w))\Big] \nonumber\\
&&+ \e \rh^2 L(Dh)\big(a+L(w)+L(Dw)\big)+\e^2\rh(1+\rh) L(h)L(Dh)+\e^2 \rh Q(h).  \nonumber 
\end{eqnarray}
Using the previous formulas now we are in position to estimate $H$, $H^2$ and $D$. With some easy computations one gets

\begin{eqnarray}
H&=&\frac{1}{\rh} \Big[a+L(w)+L(Dw)+L(D^2 w)+ \e L(Dh) (a+L(w)+L(Dw))+\e L(h)(a+L(w)+L(Dw)+L(D^2w))\Big] \nonumber\\
&&+ \e L(Dh)\big(a+L(w)+L(Dw)\big)+\e^2\frac{1}{\rh}(1+\rh) L(h)L(Dh)+\e^2 \frac{1}{\rh} Q(h).    \nonumber \\
H^2&=& \frac{1}{\rh ^2} \Big[a+L(w)+L(Dw)+L(D^2 w)+ \e (L(h)+L(Dh)+\rh L(Dh)) (a+L(w)+L(Dw)+L(D^2w))\Big] \nonumber\\
&&+\e^2\frac{1}{\rh^2}(1+\rh) L(h)L(Dh)+ \frac{\e^2}{\rh^2} (Q(h)+Q(Dh))+\frac{\e^2}{\rh}L(Dh)(L(h)+L(Dh))+\e^2Q(Dh) \nonumber\\
\det{\mathring{h}}&=& \rh ^2 \Big[a+L(w)+L(Dw)+L(D^2 w)+ \e (L(h)+L(Dh)+\rh L(Dh)) (a+L(w)+L(Dw)+L(D^2w))\Big] \nonumber\\
&&+\e^2\rh^2(1+\rh) L(h)L(Dh)+\e^2 \rh^2 (Q(h)+Q(Dh))+\e^2 \rho^3(1+\rho) Q(Dh)  \nonumber\\
D&=&\frac{\det{\mathring{h}}}{\det{\mathring{g}}}=\frac{1}{\rh ^2} \Big[a+L(w)+L(Dw)+L(D^2 w)+ \e (L(h)+L(Dh)+\rh L(Dh)) (a+L(w)+L(Dw)+L(D^2w))\Big]\nonumber\\
&& +\frac{\e^2}{\rh^2}(1+\rh) L(h)L(Dh)+\frac{\e^2} {\rh^2} (Q(h)+Q(Dh))+\e^2 \frac{1}{\rho}(1+\rho)  Q(Dh) \nonumber
\end{eqnarray}
Now we can compute $I_\e(S_p^\rho(w))=I_0(S_p^\rho(w))+\e G_1(S_p^\rho(w))+\e^2 G_2(S_p^\rho(w))$ at the second order in $\e$ and first order in $w$:
\begin{eqnarray}
I_\e(S_p^\rho))&=& \int_{S_p^\rho} \bigg[ \frac{H^2}{4}-D \bigg] d\Sigma_0=\int_{S^2}  \big[a+L(w)+L(Dw)+L(D^2 w)\big] d\Sigma_0\nonumber\\
&&+\e \int_{S^2} \Big[\big(L(h)+L(Dh)+ \rh  L(Dh)\big) \big(a+L(w)+L(Dw)+L(D^2w)\big) \Big] d\Sigma_0 \nonumber\\
&&+\e^2 \int_{S^2} \Big[  (1+\rh) L(h)L(Dh)+ \rh L(Dh)(L(h)+L(Dh))+Q(h)+Q(Dh)+\rho Q(Dh)+\rho^2  Q(Dh)  \Big] d\Sigma_0. \nonumber
\end{eqnarray}
So $G_1(S_p^\rho(w))=\int_{S^2} [(L(h)+L(Dh)+ \rh  L(Dh)\big) \big(a+L(w)+L(Dw)+L(D^2w))]$, but also
$$G_1(S_p^\rho(w))= G_1(S_p^\rho) + \int_{S^2}G_1'(S_p^\rho)w  d\Sigma_0;$$ 
with an integration by parts we get the first variation
$$\int_{S^2}G_1'(S_p^\rho)w= \int_{S^2} \Big[\Big(L(h)+ (1+\rh) \big(L(Dh)+L(D^2h)+L(D^3h) \big)\Big) w \Big] d\Sigma_0,$$
then the differential of $G_1$ at $S_p^\rho$ is
$$G_1'(S_p^\rho)= L(h)+ (1+\rh) \big[ L(Dh)+L(D^2h)+L(D^3h) \big].$$
Finally observe that 
$$G_2(S_p^\rho)=\int_{S_p^\rho} \Big[  \frac{1+\rh}{\rho^2} L(h)L(Dh)+ \frac{1}{\rh} L(Dh)(L(h)+L(Dh))+\frac{1}{\rho^2}(Q(h)+Q(Dh))+ \frac{1}{\rho} Q(Dh) + Q(Dh)  \Big] d\Sigma_0 $$
\end{proof}

\subsection{Proof of the existence Theorems}\label{subsec:ProofEx}

In order to get existence of critical points we study the reduced functional $\Phi_\e:\Rtre\oplus \R^+ \to \R$. Since for small radius $\rho$, the reduced functional coincides with the conformal Willmore functional evaluated on the perturbed geodesic spheres $S_{p,\rho}^\e(w_\e(p,\rh))$ obtained in Lemma \ref{lemEqAuxContr}, then we know the expansion of $\Phi_\e$ for small radius from Proposition \ref{svilPhipr}. Now, using the expression of the reduced functional for large radius given in Lemma \ref{lemsvile2} and the estimates of Lemma \ref{lem:G1'G2}, we are able to bound $\Phi_\e(p,\rh)$ for large radius. This is done in the following Lemma:

\begin{lem}\label{lem:CR}
Let $h_{\mu\nu}\in C^{\infty}_0(\Rtre)$ a symmetric bilinear form with compact support (it is enough that $h$ and its first derivatives decrease fast at infinity) and let $c\in \R$ such that 
$$c:=\sup\{\|h_{\mu\nu}\|_{H^1(\pi)}:\pi \text{ is an affine plane in $\Rtre$,  } \mu,\nu=1,2,3\}.$$
Then there exists a constant $C_c>0$  depending on $c$ and  $R_3>0$ such that for all $\rh>R_3$ 
$$|\Phi_\e(p,\rh)|<\e^2C_c. $$
Moreover one has that $\forall \eta>0$ there exist $\delta>0$ small enough and $R_4 \geq 0$ large enough such that for $c<\delta $ and $\rh>R_4$
$$|\Phi_\e(p,\rh)| <\eta \e^2.$$
\end{lem}

\begin{proof}
For simplicity the proof of the Lemma is done in the case $h \in C^\infty_0$. Using the notations established in Remark \ref{rem:I0''}, from Lemma \ref{lemsvile2} and Lemma \ref{lem:G1'G2} we can write the reduced functional as
\begin{eqnarray}
\Phi_\e(p,\rho)&=& \e^2 \bigg(G_2(S_p^\rh)-\frac{1}{2} \int_{S^2}\bigg[G_1'(S_p^\rh) ((I_0''(S_p^\rh))^{-1} \big[G_1'(S_p^\rh)  \big])  \bigg]\bigg)+o(\e^2)\nonumber \\
&=&\e^2\int_{S_p^\rho} \Big[  \frac{1}{\rho^2} L(h)L(Dh)+ \frac{1}{\rh} L(h) L(Dh) +\frac{1}{\rho^2}(Q(h)+Q(Dh))+ \frac{1}{\rho} Q(Dh) + Q(Dh)  \Big] d\Sigma_0\nonumber \\
&&+\e^2 \int_{S_p^\rh} \frac{1}{\rho^2}\bigg[L(h)+ (1+\rh) (L(Dh)+L(D^2h)+L(D^3h))\times \nonumber\\
&&\qquad \qquad \times  ( \triangle_{S^2}(\triangle_{S^2}+2))^{-1} \big[L(h)+ (1+\rh) (L(Dh)+L(D^2h)+L(D^3h)) \big]  \bigg] d\Sigma_0. \nonumber
\end{eqnarray}
Now denote $K=supp(h)$ which is a compact subset of $\Rtre$; of course in the formula above the domain of integration can be replaced with $S_p^\rh \cap K$. 

Observe that for all $\sigma>0$ there exists $R>0$ with the following property: 
\\for all standard spheres  $S_p^\rho$ with radius $\rho>R$ there exists an affine plane $\pi \subset \Rtre$ such that
\begin{equation}\label{eq:hd}
\|h\|_{H^1(S_p^\rh \cap K)} ^2<\|h\|_{H^1(\pi \cap K)}^2+\sigma.
\end{equation}  
This is simply because one can approximate (in $C^k$ norm for all $k \in \N$ ) the portion of standard sphere $S_p^\rho\cap K$ with a portion of an affine plane $\pi$ provided that the radius $\rh$ is large enough.

So the first integral can be bounded by a constant times $\|h\|_{H^1(\pi \cap K)}^2+\sigma$. Using the standard elliptic regularity estimates and integration by parts also the second integral can be bounded with a constant times $\|h\|_{H^1(\pi \cap K)}^2+\sigma$.
\\Hence for all $\sigma>0$ there exists $R>0$ and $\tilde{C}>0$ such that for all $(p,\rho)$ with $\rho>R$, there exists an affine plane $\pi$ such that
$$|\Phi _\e (p,\rho)|< \e ^2 \tilde{C} (\|h\|_{H^1(\pi \cap K)}^2+\sigma) .$$
Notice that $\tilde{C}$ depends on the structure of the functions $L(.)$ and $Q(.)$ but is uniform in $(p,\rho)$, $R$ and $\sigma$ as above. 
Recalling the definition of $c$ we get: 

For all $\sigma>0$ there exists $R>0$  such that for all $(p,\rho)$ with $\rho>R$,
$$|\Phi _\e (p,\rho)|<\e ^2 \tilde{C} (c^2 +\sigma) .$$
Clearly setting $\sigma=1$, $R_3=R$ and $C_c=\tilde{C} (c^2+1)$ one obtains the first part of the thesys.
For the second part we have to show that for all $\eta>0$ there exist $\delta>0$ and $R_4>0$ such that if $c<\delta$ then for all $(p,\rho)$ with $\rh>R_4$ one has $\Phi_\e(p,\rh)<\e^2 \eta$; but this is true setting above  $\delta^2=\sigma=\frac{\eta}{2\tilde{C}}$ and $R_4=R$ associated to $\sigma$ as before (observe that the estimate is uniform in $p$).
\end{proof}

Now we are in position to prove the main results of the paper.
\\\\
\begin{pfnb} {\bf of Theorem \ref{teoExist}}
In order to show the Theorem, by Lemma \ref{lemRFD}, it is enough to prove that $\Phi_\e$ has a critical point.
\\Observe that for $\rh<R_1$ 
$$\Phi_\e(p,\rh)=I_\e(S^\e _{p,\rh}(w_\e))=O(\rh^4), $$  
so $\Phi_\e$ can be extended to a $C^1$ function up to $\rho=0$ just putting $\Phi_\e(p,0)=0$  for all $p\in \Rtre$.
\\Let $R_3$ and $C_c$ be as in Lemma \ref{lem:CR}. Since $h$ has compact support, there exists a $R>0$ such that for $|p|\geq R$ and $\rh\leq R_3$, 
 $S_p^\rh \cap supp(h)=\emptyset$.
 
In order to apply the Finite Dimensional Reduction, we have to fix a compact $Z_c \subset \Rtre \oplus \R^+$. Let us choose it as
$$Z_c:=\{(p,\rh): |p|\leq R, 0\leq \rh \leq R_3 \}.$$
Apply Lemma \ref{lemw} to the compact $Z_c$ and observe that on the boundary $\partial Z_c$ we have:
\\- $\rh=0$: $\Phi_\e=0$.
\\- $|p|=R$: $\Phi_\e=0$. In fact for $|p|=R$ the standard sphere $S_p^\rh$ does not intersect the support of $h$, so $\Sigma_{p,\rh}^\e=S_p^\rh$ for all the radius $0\leq \rh \leq R_3$; since the solution of the auxiliary equation $PI_\e'(\Sigma_{p,\rh}^\e(w_\e))=0$ is unique for $w_\e$ small enough and since $S_p^\rh$ is already a critical point for $I_\e(=I_0$ since $S_p^\rh \cap supp(h)=\emptyset$) it follows that $\Sigma_{p,\rh}^\e(w_\e)=S_p^\rh$, hence
$$\Phi_\e(p,\rh)=I_\e(\Sigma_{p,\rh}^\e(w_\e))=I_\e(S_p^\rh)=I_0(S_p^\rh)=0.$$
-$\rh=R_3$: from Lemma \ref{lem:CR} we have that  $|\Phi_\e|<\e^2C_c$.

Now observe that $\Phi_\e=O(\e^2)$ uniformly on $Z_c$:
\\from the definition of reduced functional, with a Taylor expansion one gets
$$\Phi_\e(p,\rh)=I_\e(\Sigma_{p,\rh}^\e(w_\e))=I_\e'(\Sigma_{p,\rh}^\e)[w_\e]+O(\|w_\e \|^2),$$
but $\|w_\e\|_{C^{4,\a}(S^2)}=O(\e)$  and $\|v_\e\|_{C^{4,\a}(S^2)}=O(\e)$ uniformly for $(p,\rh)\in Z_c$, so
$$I_\e'(\Sigma_{p,\rh}^\e)=I_\e'(S_p^\rh(v_\e))=I_0''(S_p^\rh)[v_\e]+\e G_1'(S_p^\rh)+o(\e)=O(\e) $$
hence $\Phi_\e=O(\e^2)$ uniformly on $Z_c$.

At this moment we know that $\Phi_\e$ is of order $O(\e ^2)$ uniformly on $Z_c$ and we know its behaviour on the boundary $\partial Z_c$.
\\Now we are going to use the expansion for small radius computed in Proposition \ref{svilPhipr}. Recall that for $\rh<R_1$, $\Phi_\e(p,\rh)=I_\e(S_{p,\rh}^\e(w_\e(p,\rh)))$ and from  Proposition \ref{svilPhipr} we have the expansion:
$$\Phi_\e(p,\rh) = \frac{\pi}{5} \|S_p\|^2 \rho^4+O(\e^2) O_p(\rho ^5).$$
Recalling \eqref{def:s}, the first term can be written as $\|S_p\|^2=\e ^2 \tilde{s}_p+o(\e^2)$, so 
$$\Phi_\e(p,\rh) =  \frac{\pi}{5}\e ^2 \tilde{s}_p \rho^4+ \rho^4 o(\e^2) +O(\e^2) O_p(\rho ^5).$$
Choose $\bar{\rh}<R_1$ such that for small $\e$ the remainder $| \bar{\rho}^4 o(\e^2)+O(\bar{\rh}^5)O(\e^2)|<\e^2$  and choose $A_c>\frac{5 }{\pi}\frac{C_c+1}{\bar{\rh}^4}$. If there exists a point $\bar{p}$ such that $\tilde{s}_{\bar{p}} >A_c  $ then
$$\Phi_\e(\bar{p},\bar{\rh})>\e^2 C_c$$
so $\Phi_\e$ attains its global maximum  on $Z_c$ at an interior point $(p_{\e},\rh_{\e})$ for all $\e$ small enough and applying Lemma \ref{lemRFD} we can say that $\Sigma_{p,\rh}^\e(w_\e(p,\rh))$ is a critical point of $I_\e$ for  $\e$ small enough. 
\\ Since for  $\e\to 0$ we have $\|v_\e\|_{C^{4,\a}(S^2)} \to 0$ and $\|w_\e\|_{C^{4,\a}(S^2)} \to 0$  (see Lemma \ref{lemv} and Lemma \ref{lemw}), then the critical point $\Sigma_{p,\rh}^\e(w_\e(p,\rh))$, for small $\e$,  can be realized as normal graph on a standard sphere and it converges to a standard sphere as $\e \to 0$.
\end{pfnb}
\newline
\newline

\begin{pfnb} {\bf of Theorem \ref{teoExist2}} 
Recall \eqref{def:s} and let $\bar{p} \in \Rtre$ be a maximum point of the first term in the expansion of the squared norm of the Traceless Ricci tensor: $\tilde{s}_{\bar{p}}= M$. Observe that from Proposition \ref{svilPhipr} and from the proof of the last Theorem, for small radius $\rho$ the reduced functional $\Phi_\e(\bar{p},\rh)$ expands as
$$\Phi_\e(\bar{p},\rh)= \frac{\pi}{5}\e ^2 \tilde{s}_{\bar{p}} \rho^4+ \rho^4 o(\e^2) +O(\e^2) O_{\bar{p}}(\rho ^5).$$
Let $\bar{\rho}$ and $\e$ small enough such that the remainder $|\bar{\rho}^4 o(\e^2) + O(\e^2) O_{\bar{p}}(\bar{\rho} ^5)|<\frac{\pi}{10} M \e^2 \bar{\rho}^4$; in this way 
$$\Phi_\e(\bar{p},\bar{\rh})> \frac{\pi}{10} M \e^2 \bar{\rho}^4.$$
From the second part of Lemma \ref{lem:CR} there exist $\delta_M>0$ and $R_4>0$ such that, if $c<\delta_M$  
$$|\Phi_\e(p,\rho)|<\frac{\pi}{11} M \e^2 \bar{\rho}^4 \quad \forall (p,\rho): \rho \geq R_4.$$
(Recall that $h$ has compact support and if $\Sigma_{p,\rh}^\e(w_\e(p,\rh))$ does not intersect $supp(h)$ then $\Phi_\e(p,\rh)=0$.)

As in the proof of Theorem \ref{teoExist}, let $R>0$ be such that for $|p|\geq R$ and $\rh\leq R_4$, $S_p^\rh \cap supp(h)=\emptyset$; now we apply the Finite Dimensional Reduction to the compact subset  $Z_c \subset \Rtre \oplus \R^+$ defined as
$$Z_c:=\{(p,\rh): |p|\leq R, 0\leq \rh \leq R_4 \}.$$
If we apply Lemma \ref{lemw} to the compact $Z_c$, from the previous discussion  and from the proof of Theorem \ref{teoExist}, on the boundary $\partial Z_c$ we have:
\\- $\rh=0$: $\Phi_\e=0$.
\\- $|p|=R$: $\Phi_\e=0$.
\\-$\rh=R_4$: $|\Phi_\e(p,\rho)|<\frac{\pi}{11} M \e^2 \bar{\rho}^4$.

Observe that $(\bar{p},\bar{\rh})$ is an interior point of $\partial Z_c$ and that
$$\Phi_\e (\bar{p},\bar{\rho})>\frac{\pi}{10} M \e^2 \bar{\rho}^4>\sup _{(p,\rh)\in \partial Z_c} |\Phi_\e(p,\rho)| $$
so $\Phi_\e$ attains its global maximum  on $Z_c$ at an interior point $(p_{\e},\rh_{\e})$ for all $\e$ small enough. Applying Lemma \ref{lemRFD} we can say that $\Sigma_{p,\rh}^\e(w_\e(p,\rh))$ is a critical point of $I_\e$ for  $\e$ small enough  and we conclude as in the previous Theorem. 
\end{pfnb}


\begin{thebibliography}{99}

\bibitem[AB1]{AB1} A. Ambrosetti, M. Badiale, 
{\it Homoclinics: Poincar\'e-Melnikov type results via a variational approach,}
Ann. Inst. Henri Poincar\'e Analyse Non Linèaire 15 (1998), 233--252. 

\bibitem[AB2]{AB2} A. Ambrosetti, M. Badiale,
{\it Variational Perturbative methods and bifurcation of bound states from
the essential spectrum,} Proc. Royal Soc. Edinburgh 18 (1998), 1131--1161.

\bibitem[AM]{AM} A. Ambrosetti, A. Malchiodi, 
{\it Perturbation methods and semilinear elliptic problems in $\Rn$,} Progress in mathematics, Birkhauser (2006).


\bibitem[BK]{BK} M. Bauer, E. Kuwert, 
{\it Existence of minimizing Willmore surfaces of prescribed genus,} Int. Math. Res. Not.  10 (2003), 553--576.

\bibitem[Chen]{Chen} B.Y. Chen, 
{\it Some conformal invariants of submanifolds and their applications} Boll. Un. Mat. Ital., Ser.4, Vol. 10, (1974), 380--385.


\bibitem[ZG]{ZG} Z. Guo,
{\it Generalized Willmore functionals and related variational problems, } Diff. Geom. and its Appl. 25, (2007), 543--551.

\bibitem[GLW]{GLW} Z. Guo, H. Li, C. Wang, 
{\it The second variational formula for Willmore submanifolds in $S\sp n$,}  Results Math. Vol. 40, Num. 1-4, (2001), 205--225.

\bibitem[HL]{HL} Z. Hu, H. Li,
{\it Willmore submanifolds in a Riemannian manifold, } Cont. geom.  rel. topics ,(2004), 251--275. 


\bibitem[KS]{KS} E. Kuwert, R. Sch\"atzle,
{\it Removability of isolated singularities of Willmore surfaces} Annals of Math. Vol. 160, Num. 1, (2004), 315--357. 

\bibitem[LU]{LU} H. Li, L. Vrancken, 
{\it New examples of Willmore surfaces in $S\sp n$,} Ann. Global Anal. Geom. Vol. 23, Num. 3, (2003), 205--225.

\bibitem[MW] {MW} X. Ma, P. Wang, 
{\it Spacelike Willmore surfaces in 4-dimensional Lorentzian space forms,}  Sci. China, Ser. A, Vol. 51, Num. 9 (2008) 9, 1561--1576.

\bibitem[And-Mal]{And-Mal} A. Malchiodi, 
{\it Existence and multiplicity results for some problems in Differential Geometry,} Ph.D. Thesys, SISSA, October (2000).

\bibitem[Mon]{Mon} A. Mondino, 
{\it Some results about the existence of critical points
for the Willmore functional,}  Math. Zeit. (2009, in press).

\bibitem[MCB]{MCB} C. B. Morrey, 
{\it Multiple integrals in the calculus of variations}, Springer Verlag (1966).


\bibitem[PW]{PW} F.J. Pedit, T.J. Willmore,
{\it Conformal geometry,} Atti Sem. Mat. Fis. Univ. Modena, Vol 36, Num. 2, (1988), 237--245. 
 

\bibitem[Pet]{Pet} P. Petersen, 
{\it Riemannian Geometry, $2^{nd}$ edition} Graduate Texts in Mathematics 171, Springer (2006).


\bibitem[PX]{PX} F. Pacard, X. Xu,
{\it Constant mean curvature spheres in Riemannian manifolds, } Manuscripta Mathematica  (2008, in press).

\bibitem[Riv]{Riv} T. Rivi\`ere,
{\it Analysis aspects of Willmore surfaces,}  Invent. Math. Vol. $174$, Num. 1, (2008), 1-45.

\bibitem[SiL]{SiL} L. Simon,
{\it Existence of surfaces minimizing the Willmore functional, } Comm. Anal. Geom. Vol. $1$, Num. $2$, (1993), 281-325. 

\bibitem[Wei]{Wei} J.L. Weiner,  
{\it On a problem of Chen, Willmore, et al. } Indiana Univ. Math. J. Vol  27,  Num. 1, (1978), 19--35.

\bibitem[Will]{Will} T.J. Willmore,
{\it Riemannian Geometry, } Oxford Science Publications, Oxford University Press (1993).

\bibitem[WG] {WG} G. Wei,
{\it New examples of Willmore hypersurfaces in a sphere,} Houston J. Math. Vol. 35, Num. 1, (2009) 81-–92.

\end{thebibliography}
\end{document}